\documentclass[11pt]{amsart}
\usepackage{lmodern}
\usepackage{amsmath, amsthm, amssymb, amsfonts}
\usepackage[normalem]{ulem}
\usepackage{mathrsfs}
\usepackage{verbatim} 
\usepackage{longtable}
\usepackage{mathtools}

\usepackage{tikz}
\usetikzlibrary{decorations.pathmorphing}
\tikzset{snake it/.style={decorate, decoration=snake}}
\usepackage{caption}
\usepackage{tikz-cd}
\usetikzlibrary{arrows}
\usepackage{hyperref}

\theoremstyle{plain}
\newtheorem{thm}{Theorem}[section]
\newtheorem{cor}[thm]{Corollary}
\newtheorem{lem}[thm]{Lemma}
\newtheorem{prop}[thm]{Proposition}
\newtheorem{conj}[thm]{Conjecture}
\newtheorem{question}[thm]{Question}

\theoremstyle{definition}

\newtheorem{example}[thm]{Example}

\theoremstyle{remark}
\newtheorem{rmk}[thm]{Remark}

\newcommand{\BC}{{\mathbb{C}}}

\newcommand{\BG}{{\mathbb{G}}}

\newcommand{\BL}{{\mathbb{L}}}

\newcommand{\BP}{{\mathbb{P}}}
\newcommand{\BQ}{{\mathbb{Q}}}

\newcommand{\BZ}{{\mathbb{Z}}}

\newcommand{\CC}{{\mathcal C}}

\newcommand{\CF}{{\mathcal F}}

\newcommand{\CH}{{\mathcal H}}

\newcommand{\CL}{{\mathcal L}}
\newcommand{\CM}{{\mathcal M}}

\newcommand{\CO}{{\mathcal O}}
\newcommand{\CP}{{\mathcal P}}

\newcommand{\FF}{{\mathfrak{F}}}

\DeclareFontFamily{OT1}{rsfs}{}
\DeclareFontShape{OT1}{rsfs}{n}{it}{<-> rsfs10}{}
\DeclareMathAlphabet{\curly}{OT1}{rsfs}{n}{it}

\newcommand\Hom{\operatorname{Hom}}

\newcommand{\Chow}{\mathrm{CH}}
\newcommand{\Corr}{\mathrm{Corr}}

\newcommand{\sA}{\mathsf{A}}
\newcommand{\sH}{\mathsf{H}}

\newcommand{\mg}{\overline{\mathcal{M}}_g^{\leq 1}}
\newcommand{\mtwo}{\overline{\mathcal{M}}_2^{\mathrm{int}}}

\newcommand{\ppic}{\mathfrak{Pic}}
\newcommand{\rel}{\mathrm{rel}}

\usepackage{tikz}
\usepackage{lmodern}
\usetikzlibrary{decorations.pathmorphing}

\makeatletter
\let\@wraptoccontribs\wraptoccontribs

\begin{document}
\title[On generalized Beauville decompositions]{On generalized Beauville decompositions}
\date{\today}

\author[Y. Bae]{Younghan Bae}
\address{University of Michigan}
\email{younghan@umich.edu}



\author[D. Maulik]{Davesh Maulik}
\address{Massachusetts Institute of Technology}
\email{maulik@mit.edu}

\author[J. Shen]{Junliang Shen}
\address{Yale University}
\email{junliang.shen@yale.edu}

\author[Q. Yin]{Qizheng Yin}
\address{Peking University}
\email{qizheng@math.pku.edu.cn}

\begin{abstract}
Motivated by the Beauville decomposition of an abelian scheme and the ``Perverse = Chern'' phenomenon for a compactified Jacobian fibration, we study in this paper splittings of the perverse filtration for compactified Jacobian fibrations. 

On the one hand, we prove for the Beauville--Mukai system associated with an irreducible curve class on a $K3$ surface the existence of a Fourier-stable multiplicative splitting of the perverse filtration, which extends the Beauville decomposition for the nonsingular fibers. Our approach is to construct a Lefschetz decomposition associated with a Fourier-conjugate $\mathfrak{sl}_2$-triple, which relies heavily on recent work concerning the interaction between derived equivalences and LLV algebras for hyper-K\"ahler varieties. Motivic lifting and connections to the Beauville--Voisin conjectures are also discussed. 

On the other hand, we construct for any $g\geq 2$ a compactified Jacobian fibration of genus~$g$ curves such that each curve is integral with at worst simple nodes and the (multiplicative) perverse filtration does not admit a multiplicative splitting. Our argument relies on the recently established universal double ramification cycle relations. This shows that in general an extension of the Beauville decomposition cannot exist for compactified Jacobian fibrations even when the simplest singular point appears. 
\end{abstract}

\maketitle

\setcounter{tocdepth}{1} 

\tableofcontents
\setcounter{section}{-1}

\section{Introduction}\label{Sec0}

Throughout, we work over the complex numbers $\BC$. To motivate the discussion, we consider a principally polarized abelian scheme $\pi: A \to B$ of relative dimension $g$ over a nonsingular base $B$, so that $A$ is identified with its dual $A^\vee$.\footnote{Since we are mostly concerned with (compactified) Jacobian fibrations, we assume principal polarization to simplify some of the statements.} The Fourier--Mukai transform~\cite{Mukai}
\begin{equation}\label{FM}
\Phi_\CL: D^b\mathrm{Coh}(A) \to D^b\mathrm{Coh}(A) 
\end{equation}
associated with the normalized Poincar\'e line bundle $\mathcal{L}$ induces a (cohomological) Fourier transform~\cite{B0}
\begin{equation}\label{Fourier}
\mathfrak{F}: = \mathrm{ch}(\CL)\left( = \mathrm{exp}(c_1(\CL))\right): H^*(A, \BQ) \xrightarrow{\simeq} H^*(A, \BQ). 
\end{equation}
There is an interesting grading on $H^*(A, \BQ)$ (other than the cohomological grading) which splits the Leray filtration $L_\bullet H^*(A, \BQ)$ and interacts well with the Fourier transform (\ref{Fourier}).

\begin{thm}\cite{B, DM}\label{thm0}
There exists a splitting of the Leray filtration associated with $\pi: A \to B$,
\[
L_kH^*(A, \BQ) = \bigoplus_{i = 0}^{k} H_{(i)}^*(A, \BQ),
\]
which is
\begin{enumerate}
    \item[(a)] stable under the Fourier transform
    \[
    \mathfrak{F}\left( H_{(i)}^*(A, \BQ)\right) = H^{*}_{(2g-i)}(A, \BQ);
    \]
    \item[(b)] multiplicative with respect to the cup-product
    \[
    H_{(i)}^*(A, \BQ) \times H_{(j)}^{*}(A, \BQ) \xrightarrow{~~\cup~~} H_{(i+j)}^{*}(A, \BQ).
    \]
\end{enumerate}
\end{thm}

The splitting $H^*(A, \BQ) = \oplus_{i = 0}^{2g}H_{(i)}^*(A,\BQ)$ in Theorem \ref{thm0}, which is now referred to as the \emph{Beauville decomposition}, is induced by the ``multiplication by $N$'' map $[N]: A \to A$; more precisely, the component $H_{(i)}^*(A, \BQ)$ is given by the eigenspace 
\[
H_{(i)}^*(A, \BQ):=  \{\alpha \in H^*(A, \BQ)\,|\, [N]^* \alpha = N^{i} \alpha \}.
\]
Furthermore, it has actually been shown in \cite{B,DM} that the Beauville decomposition admits a \emph{motivic} lifting induced by algebraic cycles and cycle relations.

The work of Arinkin \cite{A1,A2} extends the Fourier--Mukai transform (\ref{FM}) to compactified Jacobian fibrations associated with families of integral projective locally planar curves. Let $C \to B$ be such a family of curves over a nonsingular base $B$, and let $\pi: \overline{J}_C \to B$ be the associated compactified Jacobian fibration.\footnote{Recall that the compactified Jacobian of an integral curve is the moduli space of rank $1$ degree $0$ torsion free sheaves on that curve; it is integral and contains an open subset parameterizing line bundles when the singularities of the curve are planar.} We further assume that the total space $\overline{J}_C$ is nonsingular. Analogously to the abelian scheme case, Arinkin's Fourier--Mukai transform induces a Fourier transform
\[
\mathfrak{F}: H^*(\overline{J}_C, \BQ) \xrightarrow{\simeq} H^*(\overline{J}_C, \BQ).
\]

The goal of this paper is to investigate the following question:
\begin{question}\label{Question}
For what class of families of integral locally planar curves $C \to B$, does there exist a generalization of the Beauville decomposition
for $\overline{J}_C$ satisfying conditions analogous to~(a) and (b) in Theorem \ref{thm0}?
\end{question}

When $C\to B$ arises as curves on a nonsingular surface $S$, we may view $\overline{J}_C$ as the moduli of certain stable torsion sheaves on $S$, and the Fourier transform $\mathfrak{F}$ governs \emph{tautological classes} on $\overline{J}_C$; see \cite[Section 5]{MSY}. Therefore, a systematic study of Question \ref{Question} may lead to a better understanding of the ``Perverse = Chern" phenomenon identifying perverse and Chern filtrations; this phenomenon was originally found for Hitchin systems via non-abelian Hodge theory and the $P=W$ conjecture \cite{dCHM1, MS_PW, HMMS} but was recently discovered in cases \cite{KPS} beyond Hitchin moduli spaces.

We will give both positive and negative results towards Question \ref{Question}. In a positive direction, our first result is the construction of a decomposition as in Theorem \ref{thm0} for the compactified Jacobian fibration $\pi: \overline{J}_C \to B$ where $C\to B= |L|$ is a complete linear system associated with an irreducible curve class on a $K3$ surface. Such a compactified Jacobian fibration is now known as the \emph{Beauville--Mukai system} \cite{Be}. Compared to Theorem \ref{thm0}, a major difference of Theorem \ref{thm0.2} is that we replace the Leray filtration by the perverse filtration \cite{dCM0}; the latter is more natural when singular fibers appear.

\begin{thm}\label{thm0.2}
Let $L$ be an irreducible curve class on a $K3$ surface $S$ with $L^2 = 2g-2$. Let $\pi: \overline{J}_C \to B=|L|$ be the associated compactified Jacobian fibration. Then there exists a splitting of the perverse filtration 
\[
P_k H^*(\overline{J}_C, \BQ) = \bigoplus_{i = 0}^{k} H_{(i)}^*(\overline{J}_C, \BQ)
\]
which is
\begin{enumerate}
    \item[(a)] stable under the Fourier transform
    \[
    \mathfrak{F}\left( H_{(i)}^*(\overline{J}_C, \BQ) \right) = H_{(2g-i)}^{*}(\overline{J}_C, \BQ);
    \]
    \item[(b)] multiplicative with respect to the cup-product
    \[
    H_{(i)}^*(\overline{J}_C, \BQ) \times H_{(j)}^{*}(\overline{J}_C, \BQ) \xrightarrow{~~\cup~~} H_{(i+j)}^{*}(\overline{J}_C, \BQ).
    \]
\end{enumerate}
\end{thm}

For a family of integral locally planar curves $C\to B$ with $\overline{J}_C$ nonsingular, we call a splitting of the perverse filtration satisfying (a, b) of Theorem \ref{thm0.2} a \emph{generalized Beauville decomposition}.

On the other hand, our next result shows that, if we work with a general compactified Jacobian fibration for curves of genus $g \geq 2$, a generalized Beauville decomposition does not exist.\footnote{For $g=1$ one can check directly that a generalized Beauville decomposition exists; \emph{c.f.}~\cite{Z2}. Therefore Theorem \ref{thm0.40} is optimal.}
Indeed, an obstruction arises even when the simplest planar singularity appears.
\begin{thm}\label{thm0.40}
For any $g\geq 2$, there exists a family of integral projective curves $C\to B$ with~$\overline{J}_C$ nonsingular such that
\begin{enumerate}
    \item[(a)] each closed fiber of $C \to B$ has at worst nodal singularities, and
    \item[(b)] the (multiplicative) perverse filtration associated with $\pi:\overline{J}_C \to B$ does not admit a multiplicative splitting.
\end{enumerate}
\end{thm}

If we weaken the splitting to the filtration, then versions of Theorem \ref{thm0.2}(a, b) were proven to hold in \cite{MSY} for \emph{any} $\pi: \overline{J}_C \to B$ with $\overline{J}_C$ nonsingular. More precisely, it is shown that the perverse filtration is multiplicative and the Fourier transform is compatible with the perverse filtration; we refer to~\cite[Theorem 2.4]{MSY} for the more precise statements.
The point of Theorem~\ref{thm0.40} is that upgrading these properties from filtrations to decompositions seems to impose strong constraints on the family. For example, the proof of Theorem \ref{thm0.2} relies heavily on recent developments in the study of derived categories of compact hyper-K\"ahler varieties. We expect that Question~\ref{Question} has a positive answer when $\pi: \overline{J}_C\to B$ is Lagrangian.\footnote{Here we do not require that $\overline{J}_C$ is compact.} We refer to Section \ref{Sec1.4} for more discussions.

As a by-product of the proof of Theorem \ref{thm0.2}, we obtain an interesting $\mathfrak{sl}_2$-triple acting on the cohomology $H^*(\overline{J}_C, \BQ)$ which is of algebraic nature; see Remark \ref{rmk2.16}.

\begin{thm}\label{thm0.4}
Let $\pi: \overline{J}_C \to B$ be as in Theorem \ref{thm0.2}. Then there is an $\mathfrak{sl}_2$-triple 
\[
(e_0, h_0, f_0) \subset \mathrm{End}\left( H^*(\overline{J}_C, \BQ) \right)
\]
satisfying that
\begin{enumerate}
    \item[(a)] it induces a generalized Beauville decomposition of $H^*(\overline{J}_C, \BQ)$, and
    \item[(b)] all three operators $e_0, h_0, f_0$ are given by relative correspondences over $B$ induced by algebraic cycles on $\overline{J}_C\times_B \overline{J}_C$.
\end{enumerate}
\end{thm}

We discuss in Section \ref{Sec3} a conjectural motivic lifting of this $\mathfrak{sl}_2$-triple and the generalized Beauville decomposition of Theorem \ref{thm0.2}. We verify the conjectures (Conjectures \ref{conj3.1} and \ref{conj3.2}) for elliptic $K3$ surfaces.

\subsection{Acknowledgements}
We would like to thank Giuseppe Ancona, Vicky Hoskins, Simon Pepin Lehalleur, Weite Pi, Shuting Shen, and Claire Voisin for helpful discussions on relevant topics. J.S.~gratefully acknowledges the hospitality of the math department of MIT during his stay in the fall of 2023 and the spring of 2024. 

Y.B.~was supported by the SNSF Postdoc.Mobility fellowship and the Deutsche Forschungsgemeinschaft (DFG, German Research Foundation) under Germany's Excellence Strategy-EXC-2047/1-390685813. J.S.~was supported by the NSF grant DMS-2301474. Q.Y.~was supported by the NSFC grants 11831013 and 11890661.

\section{Fourier transforms, generalized theta divisors, and decompositions}

In this section, we first review Arinkin's Fourier--Mukai transform on the derived category and its associated Fourier transform on cohomology. Then we discuss strategies for constructing generalized Beauville decompositions for compactified Jacobian fibrations. Along the way, we also introduce the \emph{generalized theta divisor} which provides an obstruction to the existence of a multiplicative splitting of the perverse filtration.

\subsection{Fourier transforms}\label{Sec1.1}
Throughout, we assume that $C \to B$ is a flat family of integral projective locally planar curves of arithmetic genus $g$ over a nonsingular variety $B$. Let \mbox{$\pi: \overline{J}_C \to B$} be the compactified Jacobian with $\overline{J}_C$ nonsingular. In \cite{A1, A2}, Arinkin constructed a normalized Poincar\'e sheaf 
\[
\CP \in \mathrm{Coh}\left(\overline{J}_C \times_B \overline{J}_C \right)
\] 
which induces an equivalence of the bounded derived category of coherent sheaves
\[
\Phi_\CP: D^b\mathrm{Coh}(\overline{J}_C) \xrightarrow{\simeq} D^b\mathrm{Coh}(\overline{J}_C).
\]
Its inverse is induced by the Fourier--Mukai kernel
\[
\CP^{-1} : = \CH{\mathrm{om}}\left( \CP, \CO_{\overline{J}_C\times_B \overline{J}_C} \right)\otimes p_2^* \omega_\pi[g]
\]
with $p_2: \overline{J}_C\times_B \overline{J}_C \to \overline{J}_C$ the second projection and $\omega_\pi$ the relative canonical line bundle. The functors~$\Phi_{\CP}, \Phi_{\CP^{-1}}$ yield the Fourier transforms on cohomology:
\begin{equation}\label{Fourier1}
\mathfrak{F}, \, \mathfrak{F}^{-1}: H^*(\overline{J}_C, \BQ) \xrightarrow{\simeq} H^*(\overline{J}_C, \BQ), \quad \mathfrak{F}\circ\mathfrak{F}^{-1} = \mathfrak{F}^{-1}\circ\mathfrak{F} = \mathrm{id}.
\end{equation}
More precisely, the operators $\mathfrak{F}, \mathfrak{F}^{-1}$ are induced by correspondences given by the cycle classes
\begin{equation}\label{cycles}
\mathfrak{F}= \mathrm{td}\left(-T_{\overline{J}_C\times_B\overline{J}_C} \right)^{\frac{1}{2}}\cap \tau(\CP), \quad \mathfrak{F}^{-1} = \mathrm{td}\left(-T_{\overline{J}_C\times_B\overline{J}_C} \right)^{\frac{1}{2}}\cap \tau(\CP^{-1}),
\end{equation}
where $T_{\overline{J}_C\times_B \overline{J}_C}$ is the virtual tangent bundle of the l.c.i.~scheme $\overline{J}_C\times_B \overline{J}_C$. Since in general the relative product $\overline{J}_C\times_B \overline{J}_C$ is singular and the normalized Poincar\'e sheaf $\CP$ of \cite{A2} does \emph{not} admit a finite resolution by locally free sheaves, we use the tau-class $\tau(-)$ of \cite[Chapter 18]{Ful}; see also \cite[Section 2.3]{MSY} for a brief review.

\begin{rmk}
Here the Todd convention for the Fourier transforms differs from the one in \cite{MSY}. Previously we let $\FF^{-1}$ bear all the Todd contribution so that $\FF$ restricts to the usual $\mathrm{ch}(\CL)$ over the smooth locus of $\pi: \overline{J}_C \to B$. But in the present paper we decide to let $\FF, \FF^{-1}$ share the Todd classes equally. The advantage is the compatibility with the pushforward under the closed embedding~$i: \overline{J}_C \times_B \overline{J}_C \hookrightarrow \overline{J}_C \times \overline{J}_C$, \emph{i.e.},
\[
i_*\FF = \mathrm{td}\left(T_{\overline{J}_C \times \overline{J}_C}\right)^{\frac{1}{2}}\mathrm{ch}(i_*\CP), \quad i_*\FF^{-1} = \mathrm{td}\left(T_{\overline{J}_C \times \overline{J}_C}\right)^{\frac{1}{2}}\mathrm{ch}(i_*\CP^{-1}).
\]
In particular, the Fourier transforms (\ref{Fourier1}) on $H^*(\overline{J}_C, \BQ)$ agree with the ones induced by the Fourier--Mukai kernels on the nonsingular \emph{absolute} product $\overline{J}_C\times \overline{J}_C$ supported on the closed subset $\overline{J}_C\times_B \overline{J}_C$. We emphasize that the Fourier transforms are naturally relative correspondences over $B$. This will be used to show Theorem \ref{thm0.4}(b) for the operators $f_0, h_0$. We will discuss its motivic lifting in Section \ref{Sec3}, where all the constructions are relative over $B$.
\end{rmk}

\subsection{Decompositions}\label{Sec1.2}

If the family of curves $C \to B$ has no singular fiber, the Jacobian fibration $\pi: J_C \to B$ is a principally polarized abelian scheme over $B$. There are three typical approaches to recover the Beauville decomposition. In the following, we discuss these approaches, and the difficulty in extending each approach over singular fibers.

The first approach, as described in Section \ref{Sec0}, is to use the ``multiplication by $N$'' map $[N]: J_C \to J_C$ to decompose the total cohomology into eigenspaces. However, when singular fibers appear, it is difficult to extend the map $[N]$ over them. 

The second approach, which was initiated in \cite{B,DM}, is to use the cycles (\ref{cycles}) associated with the Fourier transforms to obtain \emph{projectors}, which yields a (motivic) decomposition of the total cohomology. When there are singular fibers, this idea has been exploited in \cite{MSY} intensively. It yields a number of consequences concerning the perverse filtration, which include the construction of a motivic decomposition splitting the perverse filtration analogously to the smooth case. However, as discussed in \cite[Section 2.5.5]{MSY}, singular fibers force us to modify the correspondences constructed from $\mathfrak{F}, \mathfrak{F}^{-1}$ in order to get projectors. Consequently, the resulting decomposition becomes mysterious and is therefore hard to analyze; in particular it is unclear whether it is multiplicative and how it interacts with the Fourier transform.

The third approach is due to K\"unnemann \cite{Kun}, who constructed an $\mathfrak{sl}_2$-triple $(e_0,f_0,h_0)$ via the Fourier transform acting on the cohomology of $J_C$ and showed that the weight decomposition associated with this $\mathfrak{sl}_2$-triple recovers the Beauville decomposition. More precisely, the operator $e_0$ is given by the cup-product with a relatively ample class, the operator $f_0$ is the Fourier-conjugate of $e_0$, and $h_0$ is the commutator of $e_0, f_0$. The advantage of this approach is that all three operators are constructed for $\pi: \overline{J}_C \to B$ with singular fibers, so it suffices to check commutation relations between them. 

The goal of this paper is to explore the third approach above for compactified Jacobian fibrations. We first discuss a proposal for constructing the $\mathfrak{sl}_2$-triple for general $C \to B$, and a natural obstruction for this approach to work. This also provides an obstruction to the existence of a multiplicative splitting of the perverse filtration. 

\subsection{Generalized theta divisors}\label{Sec1.3}

Now we consider $C\to B$ with singular fibers. By the relative Hard Lefschetz theorem, a relatively ample class $\omega$ has perversity $2$,
\[
\omega \in P_2 H^2(\overline{J}_C, \BQ) \setminus P_1H^2(\overline{J}_C, \BQ).
\]
Therefore, if the generalized Beauville decomposition exists, there is a relatively ample class~$\Theta$ given by the projection of $\omega$ to the component $H_{(2)}^2(\overline{J}_C, \BQ)$. Since the perverse filtration terminates at $P_{2g}H^*(\overline{J}_C, \BQ)$, we have
\[
H^{*}_{(i)}(\overline{J}_C, \BQ) = 0, \quad i> 2g.
\]
Therefore the condition $\Theta \in H^2_{(2)}(\overline{J}_C, \BQ)$ forces the vanishing
\begin{equation}\label{eqn5}
\Theta ^{g+1} = 0 \in H^{2g+2}(\overline{J}_C, \BQ).
\end{equation}
Over the smooth locus of $\pi$, such a $\Theta$ is provided by the (normalized) relative theta divisor. We call a relatively ample $\BQ$-divisor $\Theta$ satisfying (\ref{eqn5}) a \emph{generalized theta divisor}.

\begin{question}\label{Q1.2}
    Does there exist a generalized theta divisor for $\pi: \overline{J}_C \to B$?
\end{question}

\begin{rmk}
    If the answer to Question \ref{Q1.2} is negative for $C \to B$, then there exists \emph{no} multiplicative splitting of the perverse filtration $P_\bullet H^*(\overline{J}_C, \BQ)$.
\end{rmk}

We will show in Section \ref{Sec4} that a generalized theta divisor does not exist in general for~$C \to B$ whose fibers have at worst simple nodes when $g\geq 2$.

\subsection{A proposal for constructing the $\mathfrak{sl}_2$-triple}\label{sec1.3}

We assume Question \ref{Q1.2} has an affirmative answer for a family $C\to B$ with $\Theta \in H^2(\overline{J}_C, \BQ)$ a generalized theta divisor. Following~\cite{Kun}, we define the raising operator
\[
e_0: =  \Theta \cup - \in \mathrm{End}\left( H^*(\overline{J}_C, \BQ) \right),
\]
and its Fourier conjugate
\[
f_0:=  -\mathfrak{F}^{-1}\circ e_0 \circ \mathfrak{F} \in \mathrm{End}\left( H^*(\overline{J}_C, \BQ) \right).
\]
We write $h_0$ to be their commutator
\[
h_0 := [e_0, f_0] \in \mathrm{End}\left( H^*(\overline{J}_C, \BQ) \right).
\]
The operator $e_0$ is of degree $2$ with respect the cohomological grading, but \emph{a priori} the operators $f_0, h_0$ may be of mixed degrees.

\begin{question}\label{Q1.3}
    Do $(e_0, h_0, f_0)$ form an $\mathfrak{sl}_2$-triple with $f_0$ of degree $-2$?
\end{question}

Both the $\mathfrak{sl}_2$-triple and the degree $-2$ conditions concern a set of cohomology relations. If the answer to Question \ref{Q1.3} is affirmative for $C\to B$, we get a decomposition from this $\mathfrak{sl}_2$-triple. Finally, we require that the decomposition obtained from the $\mathfrak{sl}_2$-triple $(e_0, h_0, f_0)$ satisfies all the desired properties --- it is Fourier-stable, splits the perverse filtration, and is multiplicative with respect to the cup-product. 

\begin{question}\label{Q1.4}
    Is the decomposition associated with the $\mathfrak{sl}_2$-triple $(e_0, h_0, f_0)$ a Fourier-stable multiplicative decomposition splitting the perverse filtration?
\end{question}

In general, answering the questions above for a given family $C\to B$ seems to be difficult. As we show in Section \ref{Sec4}, we do not expect positive answers in general. On the other hand, we are able to provide positive answers to all three questions in the setting of the Beauville--Mukai system, where we use recently developed tools in the study of compact hyper-K\"ahler varieties. As we will discuss in Section \ref{Sec2}, Question \ref{Q1.2} is solved using the Beauville--Bogomolov--Fujiki quadratic form, Question \ref{Q1.3} is solved using the Looijenga--Lunts--Verbitsky (LLV) Lie algebra \cite{LL, Ver95} and Taelman's recent work \cite{Taelman} on its interaction with derived equivalences, and Question \ref{Q1.4} is solved by relating with a decomposition introduced in \cite{SY}.

\subsection{Relations to other work}\label{Sec1.4}

As a consequence of the $P=W$ conjecture \cite{dCHM1} (now a theorem by \cite{MS_PW, HMMS}) and the Hodge--Tate decomposition of the character variety \cite{Shende}, the perverse filtration associated with the Hitchin fibration $h: M_{\mathrm{Higgs}} \to B$ has a multiplicative splitting. The idea of \cite{HMMS} is that this splitting is characterized by an $\mathfrak{sl}_2$-triple constructed geometrically by Hecke operators. In \cite{MSY}, a different proof of the $P=W$ conjecture was given using the Fourier transform. As commented in \cite{Hoskins}, considerations from the geometric Langlands correspondence suggest that the Fourier transform should exchange Hecke operators and certain tautological operators. Therefore, the Fourier transform seems to provide a path to connect the two proofs \cite{MS_PW, HMMS}, which concern tautological operators and Hecke operators respectively. 

However, as we explained in Section \ref{Sec1.2}, the Fourier transform does not seem to provide the desired decomposition using the approach of Deninger--Murre \cite{DM} when there are singular fibers. On the other hand, results of this paper suggest that K\"unnemann's Lefschetz decomposition may have a chance to work over singular fibers in the \emph{Lagrangian} setting. Our method here relies on techniques from compact hyper-K\"ahler geometry, so it only works for compact total spaces $\overline{J}_C$. New ideas may be needed to treat other interesting cases of Lagrangian fibrations $\pi: \overline{J}_C \to B$. A typical example to be understood is that $B$ is the elliptic locus of the Hitchin base with $\pi: \overline{J}_C \to B$ is the corresponding Hitchin system; we expect that in the Hitchin setting the operators $e_0, h_0, f_0$ constructed as in this paper form an $\mathfrak{sl}_2$-triple, which yields a generalized Beauville decomposition. 

In another direction, we propose in Section \ref{Sec3} conjectures on lifting the $\mathfrak{sl}_2$-triple and the generalized Beauville decomposition \emph{motivically}. This is related to the Beauville--Voisin conjectures \cite{Bea, Voi, Voi23} concerning algebraic cycles on compact hyper--K\"ahler varieties. 

\section{Proof of Theorem \ref{thm0.2}}\label{Sec2}


In this section we prove Theorem \ref{thm0.2} following the proposal of Section \ref{sec1.3}. We consider a~pair
\[
(S,L), \quad  L^2 = 2g-2
\]
with $S$ a projective $K3$ surface and $L \in \mathrm{Pic}(S)$ an irreducible curve class, that is, every curve in the linear system $|L|$ is irreducible. 


\subsection{Compact hyper-K\"ahler varieties}
Since the family of curves $C\to B$ is given by the complete linear system associated with $|L|$, the compactified Jacobian $\overline{J}_C$ is the moduli of stable $1$-dimensional sheaves $\CF$ on the $K3$ surface $S$ with
\[
[\mathrm{supp}(\CF)] = L, \quad \chi(\CF) = 1-g,
\]
with the stability condition given by any polarization. In particular, $\overline{J}_C$ is a compact hyper-K\"ahler variety of $K3^{[g]}$-type, and $\pi: \overline{J}_C \to B$ is a Lagrangian fibration with $B \simeq \BP^g$. This allows us to apply tools in compact hyper-K\"ahler geometry to study the cohomology of $\overline{J}_C$, which lead to solutions to Questions \ref{Q1.2}, \ref{Q1.3}, and \ref{Q1.4}.

\subsection{Solving Question \ref{Q1.2}}

By specialization, it suffices to solve Question \ref{Q1.2} for a very general pair $(S, L)$. In this case, the moduli space $\overline{J}_C$ is of Picard rank $2$; the algebraic part of the cohomology $H^2(\overline{J}_C, \BQ)$ is spanned by an ample class $\sA$ and the pullback of the hyperplane class~$\sH$ of the base $B \simeq \BP^g$.

We set $\Theta = \sA + \lambda \sH$ for some $\lambda \in \BQ$, and we are looking for a solution to the following equation
\[
\Theta^{g + 1} = (\sA + \lambda  \sH)^{g+1} = 0.
\]
Although this is a polynomial equation of degree $g$ (this is because $\sH^{g+1}=0$), the Beauville--Bogomolov--Fujiki (BBF) quadratic form $q(-)$ on $H^2(\overline{J}_C, \BQ)$ \cite{B1} reduces it to a \emph{linear} equation:
\[
q(\sA + \lambda \sH) =q(\sA) + 2 \lambda (\sA,\sH) =0.
\]
Here $(-,-)$ is the bilinear form associated with the BBF form, and we have used that \mbox{$q(\sH)=0$}. This solves $\lambda$ uniquely:
\[
\lambda = -\frac{q(\sA)}{2(\sA, \sH)} \in \BQ.
\]

Now we consider arbitrary $(S,L)$ with $L$ an irreducible curve class. After changing the basis and specialization, we may assume that the algebraic part of $H^2(\overline{J}_C, \BQ)$ contains a relatively ample class $\Theta$ which, together with $\sH$, forms a rank $2$ hyperbolic $\BQ$-lattice under the BBF~form,
\[
q(\Theta) = q(\sH) =0, \quad (\Theta, \sH) = 1.
\]
As in Section \ref{sec1.3}, we define the operators
\begin{equation}\label{operators}
e_0:=  \Theta \cup -, \quad f_0:= -\mathfrak{F}^{-1}\circ e_0 \circ \mathfrak{F}, \quad h_0 := [e_0, f_0] \in \mathrm{End}\left(H^*(\overline{J}_C, \BQ) \right).
\end{equation}
Clearly, all three operators are relative algebraic correspondences over $B$.

The goal of the next few sections is to establish the following proposition.

\begin{prop}\label{prop2.1}
The operator $f_0$ is of degree $-2$, and $(e_0, h_0, f_0)$ form an $\mathfrak{sl}_2$-triple in $\mathrm{End}\left(H^*(\overline{J}_C, \BQ) \right)$.
 \end{prop}

Our main tools are the LLV algebra \cite{LL, Ver95} and its categorical perspective initiated by Taelman \cite{Taelman, Beckmann, Markman2}.

\subsection{LLV algebras}

We first recall the definition of the LLV algebra. Let $X$ be a compact hyper-K\"ahler variety of dimension $2n$. Since it suffices to prove the statements of Proposition~\ref{prop2.1} in $\mathrm{End}\left(H^*(\overline{J}_C, \BC) \right)$, we work with LLV algebras with $\BC$-coefficients for convenience.

An element $\eta \in H^2(X, \BC)$ is called of \emph{Lefschetz type} if for any $k \geq 0$, the cup-product with~$\eta^k$ induces an isomorphism
\[  \eta^k\cup -  : H^{2n - k}(X, \BC) \xrightarrow{\simeq} H^{2n + k}(X, \BC).
\]
In other words, the class $\eta$ induces an $\mathfrak{sl}_2$-triple $(e_\eta, h, f_\eta)$ acting on $H^*(X, \BC)$. The LLV algebra of $X$, denoted by $\mathfrak{g}(X)$, is defined to be the Lie algebra generated by all~$\mathfrak{sl}_2$-triples associated with Lefschetz type classes. 

The LLV algebra of a compact hyper-K\"ahler variety can be calculated in terms of the extended Mukai lattice \cite{LL, Ver95}. Recall that
the (complexified) extended Mukai lattice of $X$ is a graded vector space
\[
\widetilde{H}(X,\BC): = \BC\alpha \oplus H^2(X, \BC) \oplus \BC\beta, \quad \mathrm{deg}(\alpha) =0, \quad \mathrm{deg}(\beta) = 4
\]
with a quadratic form $\widetilde{q}(-)$. The form $\widetilde{q}(-)$ restricts to the BBF form on $H^2(X, \BC)$, and the two classes $\alpha, \beta$ are orthogonal to $H^2(X,\BC)$ satisfying
\[
\widetilde{q}(\alpha)=\widetilde{q}(\beta)=0, \quad (\alpha, \beta) = -1.
\]
Then the results of \cite{LL, Ver95} give a canonical isomorphism
\[
\mathfrak{g}(X) = \mathfrak{so}(\widetilde{H}(X,\BC)).
\]


\begin{example}[The symplectic $\mathfrak{sl}_2$-triple]\label{example}
We describe an interesting $\mathfrak{sl}_2$-triple in $\mathfrak{g}(X)$ which plays a crucial role in \cite{Ver95}. The holomorphic symplectic form $\sigma$ on $X$ produces a class \mbox{$\sigma \in H^2(X, \BC)$}. It further induces an operator 
\[
e_{\sigma}:= \sigma \cup - \in  \mathrm{End}( H^*(X, \BC) ).
\]
We denote by $h_{\mathrm{hol}}$ the holomorphic grading operator, which is characterized by acting on $H^{p,q}(X)$ via $h_\mathrm{hol} = (p-n)\mathrm{id}$. Then $e_{\sigma}, h_{\mathrm{hol}}$ can be completed into an $\mathfrak{sl}_2$-triple
\[
(e_\sigma,  h_{\mathrm{hol}}, f_{\sigma}) \subset \mathfrak{g}(X),   
\]
which we call the \emph{symplectic $\mathfrak{sl}_2$-triple} associated with $\sigma$.
\end{example}

\subsection{Lie algebra calculations}\label{sec2.4}
In this section, we carry out some explicit calculations with the LLV algebra to be used later. We consider four classes $\eta_1,\eta_2,\eta_3,\eta_4 \in H^2(X, \BC)$ satisfying 
\begin{equation} \label{eq:condition2.4}
\widetilde{q}(\eta_1) = \widetilde{q}(\eta_2) = \widetilde{q}(\eta_3) = \widetilde{q}(\eta_4) \neq 0, \quad (\eta_i, \eta_j) =0, \quad i\neq j.
\end{equation}
They generate $\mathfrak{sl}_2$-triples $(e_{\eta_i}, h, f_{\eta_i})$ in the LLV algebra $\mathfrak{g}(X)$ (this is essentially deduced in~\cite{Ver95}; see \emph{e.g.}~\cite[Lemma 2.5]{SY}).

We write
\[
K_{ij}: = [e_{\eta_i}, f_{\eta_j}] \in \mathfrak{g}(X), \quad i\neq j.
\]
Recall the Verbitsky relations concerning the operators $K_{ij}$.

\begin{prop}[Verbitsky relations \cite{Ver95}]\label{prop2.3}
For three distinct integers $1\leq i,j,k \leq 4$, we have
\begin{align*}
    K_{ij} & =-K_{ji}, \quad  [K_{ij}, K_{jk}]= 2K_{ik}, \quad [K_{ij}, h]=0, \\
    \quad [K_{ij}, e_{\eta_j}] & = 2e_{\eta_i},\quad  [K_{ij}, f_{\eta_j}]=2f_{\eta_i}, \quad [K_{ij}, e_{\eta_k}] = [K_{ij}, f_{\eta_k}] = 0. 
\end{align*}
\end{prop}

Now we define for $i\neq j$ the following elements
\[
\sigma_{ij}:= \frac{1}{2}({\eta_i} +\mathfrak{i} {\eta_j}),\quad \overline{\sigma}_{ij}:= \frac{1}{2}({\eta_i} -\mathfrak{i} {\eta_j}) \in H^2(X, \BC)
\]
which induce the following elements in the LLV algebra:
\[
e_{\sigma_{ij}} = \frac{1}{2} (e_{\eta_i} +\mathfrak{i} e_{\eta_j}),\quad e_{\overline{\sigma}_{ij}} = \frac{1}{2} (e_{\eta_i} -\mathfrak{i} e_{\eta_j})\in \mathfrak{g}(X).
\]
One can check directly that
\[
\widetilde{q}(\sigma_{ij}) = \widetilde{q}(\overline{\sigma}_{ij})=0, \quad (\sigma, \overline{\sigma}) \neq 0.
\]

\begin{prop}[{\emph{c.f.}~\cite[Section 3.1]{SY}}] \label{prop2.4}
For distinct $i,j$, we consider the elements 
\[
f_{\sigma_{ij}}: = \frac{1}{2}(f_{\eta_i} - \mathfrak{i} f_{\eta_j}), \quad f_{\overline{\sigma}_{ij}}: = \frac{1}{2}(f_{\eta_i} + \mathfrak{i} f_{\eta_j})\in \mathfrak{g}(X).
\]
The following hold.
\begin{enumerate}
    \item[(a)] The three elements $(e_{\sigma_{ij}},h_{\sigma_{ij}}:= [e_{\sigma_{ij}}, f_{\sigma_{ij}}], f_{\sigma_{ij}})$ in $\mathfrak{g}(X)$ form an $\mathfrak{sl}_2$-triple.
    \item[(b)] The three elements $(e_{\overline{\sigma}_{ij}},h_{\overline{\sigma}_{ij}}:= [e_{\overline{\sigma}_{ij}}, f_{\overline{\sigma}_{ij}}], f_{\overline{\sigma}_{ij}})$ in $\mathfrak{g}(X)$ form an $\mathfrak{sl}_2$-triple.    \item[(c)] We have
    \[
    [e_{\sigma_{ij}}, f_{\overline{\sigma}_{ij}}]=  [e_{\overline{\sigma}_{ij}}, f_{{\sigma}_{ij}}]  = 0.
    \]
\end{enumerate}
\end{prop}

\begin{proof}
All three statements follow from the Verbitsky relations of Proposition \ref{prop2.3}.

Alternatively, we notice that when $\eta_i, \eta_j, \eta_k$ form the K\"ahler classes associated with three complex structures, the classes $\sigma_{ij}, \overline{\sigma}_{ij}$ are given by the symplectic form $\sigma$ and its complex conjugate with respect to the complex structure corresponding to $\eta_k$. Therefore (a, b, c) follow from the fact that $\sigma,\overline{\sigma}$ generate an $\mathfrak{sl}_2\times \mathfrak{sl}_2 \subset \mathfrak{g}(X)$ given by Example \ref{example} and its complex conjugate; the general case can be reduced to this special case using hyper-K\"ahler rotations as in \cite{Ver95}.
\end{proof}

For our purpose, we need to construct a more complicated $\mathfrak{sl}_2$-triple as follows which will play a crucial role. We set
\[
L:= [e_{\sigma_{12}}, f_{\sigma_{34}}], \quad \Lambda:= [e_{\sigma_{34}}, f_{\sigma_{12}}], \quad  H:= [L, \Lambda] \in \mathfrak{g}(X).
\]

\begin{prop}\label{prop2.5}
    The elements $(L, H, \Lambda)$ form an $\mathfrak{sl}_2$-triple in $\mathfrak{g}(X)$.
\end{prop}

\begin{proof}
    It suffices to show that 
    \begin{equation}\label{strange_sl2}
    [H, L]=2L, \quad [H, \Lambda] = -2\Lambda.
    \end{equation}
This can also be deduced from the Verbitsky relations. Since the calculation is more involved than that of Proposition \ref{prop2.4}, we provide a proof for the reader's convenience.

We first express $L, \Lambda, H$ in terms of $K_{ij} ~(i<j)$ using the Verbitsky relations:
\begin{align*}
    L &= \frac{1}{4}( (K_{13}+K_{24}) - \mathfrak{i} (K_{14}-K_{23})), \\
    \Lambda &= \frac{1}{4}( - (K_{13}+K_{24}) - \mathfrak{i}(K_{14} - K_{23})),\\
 H &= [L, \Lambda] =  -\frac{\mathfrak{i}}{8} [K_{13}+K_{24}, K_{14} - K_{23}  ] = -\frac{\mathfrak{i}}{2}(K_{12} - K_{34}),
\end{align*}
where we have used $K_{ij}=-K_{ji}$, $[K_{ij}, K_{jk}]=2K_{ik}$ in calculating $H$. We further notice that
\begin{align*}
    [K_{12} - K_{34}, K_{13}+K_{24}] = 4(K_{14}-K_{23}),\\
    [K_{12} - K_{34}, K_{14}-K_{23}] = -4(K_{13}+K_{24}),
\end{align*}
from which we may calculate the commutators $[H,L]$, $[H, \Lambda]$. The commutation relations (\ref{strange_sl2}) then follow.
\end{proof}

\begin{rmk}
Unlike the $\mathfrak{sl}_2$-triples of Proposition \ref{prop2.4}, the $\mathfrak{sl}_2$-triple of Proposition \ref{prop2.5} does not have a clear geometric meaning. However, we will use this $\mathfrak{sl}_2$-triple to construct a Fourier--conjugate $\mathfrak{sl}_2$-triple in Section \ref{Sect2.7}.
\end{rmk}

\subsection{Derived perspectives of LLV algebras}
A beautiful observation of Taelman \cite{Taelman} is that $\mathfrak{g}(X)$ is a derived invariant --- it only depends on the bounded derived category $D^b\mathrm{Coh}(X)$. Taelman's result relies on the fact that the LLV algebra can be recovered from Hochschild cohomology which we review as follows.

We denote by $\mathrm{HT}^*(X)$ the graded $\BC$-algebra of polyvector fields with degree $k$ part
\[
\mathrm{HT}^k(X) = \bigoplus_{p+q=k} H^q\left(X, \wedge^p T_X \right).
\]
Via the isomorphism $\wedge^p T_X \simeq \Omega^p_X$ induced by $\sigma$, and the natural isomorphism between singular cohomology and de Rham cohomology, we have a graded isomorphism
\[
\mathrm{HT}^*(X) = H^*(X,\BC).
\]
The cohomology $H^*(X, \BC)$ admits a natural action of the algebra $\mathrm{HT}^*(X)$ by contraction, \emph{i.e.}, for $v\in H^q(X, \wedge^p T_X)$ and $\gamma \in H^{p',q'}(X)$, we have
\[
v \lrcorner \gamma \in H^{p'-p, q'+q}(X).
\]
Using this action, we consider $\mathfrak{sl}_2$-triples
\begin{equation}\label{HH_sl2}
(e'_{\mu}, h', f'_{\mu}), \quad \mu \in \mathrm{HT}^2(X)
\end{equation}
in the Lie algebra $\mathrm{End}( H^*(X, \BC) )$, where $e'_\mu$ acts by contracting $\mu$ as above and $h'$ acts on~$H^{p,q}(X)$ via $(q-p)\mathrm{id}$. Taelman proved in \cite{Taelman} that $\mathfrak{g}(X)$ is exactly the minimal Lie subalgebra of $\mathrm{End}( H^*(X, \BC) )$ containing all $\mathrm{sl}_2$-triples of the form (\ref{HH_sl2}). Since the Hochschild--Kostant--Rosenberg isomorphism identifies the Hochschild cohomology of $X$ with $\mathrm{HT}^*(X)$, the following theorem is a consequence of the discussion above.

\begin{thm}[\cite{Taelman}]\label{thm2.2}
A derived equivalence between compact hyper-K\"ahler varieties
\[
\Phi: D^b\mathrm{Coh}(X) \xrightarrow{\simeq} D^b\mathrm{Coh}(X')
\]
induces a natural isomorphism of Lie algebras
\[
\Phi^\mathfrak{g}: \mathfrak{g}(X) \xrightarrow{\simeq} \mathfrak{g}(X').
\]
Moreover, the cohomological correspondence $\Phi^{H}: H^*(X, \BC)\xrightarrow{\simeq} H^*(X', \BC)$ induced by $\Phi$ is equivariant with respect to $\Phi^\mathfrak{g}$.
\end{thm}

\subsection{Fourier transforms}

Now we focus on the the compact hyper-K\"aher variety $\overline{J}_C$ and Arinkin's Fourier--Mukai transform
\begin{equation}\label{FMP}
\Phi_\CP: D^b\mathrm{Coh}(\overline{J}_C) \xrightarrow{\simeq} D^b\mathrm{Coh}(\overline{J}_C).
\end{equation}
Applying Theorem \ref{thm2.2}, we obtain an isomorphism of Lie algebras
\[
\mathfrak{F}^\mathfrak{g}: \mathfrak{g}(\overline{J}_C) \xrightarrow{\simeq} \mathfrak{g}(\overline{J}_C)
\]
with respect to which the Fourier transform $\mathfrak{F} : H^*(\overline{J}_C, \BC) \xrightarrow{\simeq} H^*(\overline{J}_C, \BC)$ is equivariant.

Recall the operators (\ref{operators}); we note that they all lie in the LLV algebra.

\begin{lem}
We have 
\[
e_0, f_0, h_0 \in \mathfrak{g}(\overline{J}_C).
\]
\end{lem}

\begin{proof}
By Verbitsky's construction of the LLV algebra \cite{Ver95}, any element $e_\eta$ associated with $\eta \in H^2(X, \BC)$ lies in the LLV algebra. Therefore, $e_0 \in \mathfrak{g}(\overline{J}_C)$. Alternatively, this can be seen from \cite[Lemma 2.5]{SY}, since any class $\eta \in H^2(\overline{J}_C, \BC)$ can be written as the $\eta= \eta' - \eta''$ with~$q(\eta), q(\eta')$ nonzero. By definition and Theorem \ref{thm2.2}, we have
\[
f_0 = -\mathfrak{F}^\mathfrak{g}(e_0) \in \mathfrak{g}(\overline{J}_C),\] 
which further implies $h_0 = [e_0, f_0] \in \mathfrak{g}(\overline{J}_C)$.
\end{proof}

Next, we want to show that $(e_0,h_0,f_0)$ form an $\mathfrak{sl}_2$-triple in $\mathfrak{g}(\overline{J}_C)$, and we relate it to the symplectic $\mathfrak{sl}_2$-triple of Example \ref{example}.

Using the Verbitsky component generated by classes in the second cohomology, Taelman attaches to the derived equivalence (\ref{FMP}) a canonical (Hodge) isometry between the extended Mukai lattices
\begin{equation}\label{isom}
\mathfrak{F}^{\widetilde{H}}: \widetilde{H}(\overline{J}_C, \BC) \xrightarrow{\simeq} \widetilde{H}(\overline{J}_C, \BC);
\end{equation}
see \cite[Theorems 4.10, 4.11]{Taelman}. Taelman's construction relies on certain numerical assumptions which are satisfied by all known examples of compact hyper-K\"ahler varieties including $\overline{J}_C$.





Using extended Mukai vectors, the action of (\ref{isom}) on the four algebraic classes
\begin{equation}\label{alg}
\alpha, \beta, \Theta, \sH \in \widetilde{H}(\overline{J}_C, \BC)
\end{equation}
is described by the following formulas of Beckmann.

\begin{prop}[\cite{Beckmann}]\label{prop2.9}
There is a constant $c_0 = 1$ for $g$ odd, and $c_0 = 1$ or $-1$ for $g$ even, such~that
\begin{gather*}
    \mathfrak{F}^{\widetilde{H}}(\alpha) = -c_0\left(\Theta - \frac{g+1}{2}\beta\right), \quad 
\mathfrak{F}^{\widetilde{H}}(\beta) =  c_0\sH,\\ 
\mathfrak{F}^{\widetilde{H}}(\Theta) = c_0\left(\alpha -\frac{g+1}{2}\sH\right), \quad \mathfrak{F}^{\widetilde{H}}(\sH) = -c_0\beta.
\end{gather*}
\end{prop}

\begin{proof}
    The action of (\ref{isom}) on the classes (\ref{alg}) is established in \cite[Section 10.2]{Beckmann}.\footnote{Here we corrected some signs in \cite[Section 10.2]{Beckmann} up to a global $\pm1$; the method of \cite{Beckmann} does not determine the precise sign for $g$ even.} The calculation uses the theory of extended Mukai vectors \cite{Beckmann, Markman2}, which further relies on the Rozansky--Witten theory governing the square root of the Todd class (see \cite[Section 3]{Beckmann}). For our purpose, it suffices to check that the class $\frac{1}{2}\lambda -\frac{g-1}{4}f$ in Beckmann's calculation coincides with~$\Theta$; this class actually is uniquely characterized by that it and $\sH$ form a hyperbolic lattice when $(S,L)$ is very general.
\end{proof}

We also remark that the class of the symplectic form $\sigma \in \widetilde{H}(\overline{J}_C, \BC)$ is fixed up to a sign by the action of \eqref{isom}.

\begin{prop} \label{fixsigma}
There is a constant $c_1 = 1$ for $g$ even, and $c_1 = 1$ or $-1$ for $g$ odd, such~that
\[
\mathfrak{F}^{\widetilde{H}}(\sigma) = c_1\sigma.
\]
\end{prop}

\begin{proof}
We consider the composition $\Phi_\CP \circ \Phi_\CP$ and its induced action on the extended Mukai lattice, which equals $\mathfrak{F}^{\widetilde{H}} \circ \mathfrak{F}^{\widetilde{H}}$ by the functoriality of \cite[Theorems 4.10 and 4.11]{Taelman}. Recall from~\cite{A2} that
\[
\Phi_\CP \circ \Phi_\CP = \nu^*(-) \otimes \pi^*\omega_B[-g]
\]
where $\nu: \overline{J}_C \to \overline{J}_C$ is the \emph{antisymplectic} involution induced by fiberwise duality. In particular, for $g$ odd we have
\begin{equation*} \label{eq:FFsigma}
\mathfrak{F}^{\widetilde{H}} \circ \mathfrak{F}^{\widetilde{H}} (\sigma) = \sigma.
\end{equation*}
The same holds for $g$ even which is due to the extra orientation appearing in \cite[Theorem~4.11]{Taelman}. In fact, by specialization we may assume $(S, L)$ very general so that the algebraic part of~$H^2(\overline{J}_C, \BC)$ is of rank $2$. The derived equivalence $\Phi_\CP \circ \Phi_\CP$ sends the structure sheaf~$\mathcal{O}_{\overline{J}_C}$ to a line bundle shifted by an even number, which after \cite[Lemma 4.2]{Beckmann} implies
\[
\mathrm{det}(\mathfrak{F}^{\widetilde{H}} \circ \mathfrak{F}^{\widetilde{H}}) = 1.
\]
This forces $\mathfrak{F}^{\widetilde{H}} \circ \mathfrak{F}^{\widetilde{H}}$ to act as $1$ (and not $-1$) on the transcendental part of $\widetilde{H}(\overline{J}_C, \BC)$ since the latter is of odd rank ($ = 21$).

For now we have $\mathfrak{F}^{\widetilde{H}}(\sigma) = \pm\sigma$. To see that the sign is positive for $g$ even, we repeat the argument above with $\Phi_\CP$. Again by \cite[Lemma 4.2]{Beckmann} we have
\[
\mathrm{det}(\mathfrak{F}^{\widetilde{H}}) = 1
\]
since $\Phi_\CP$ sends a sufficiently ample line bundle on $\overline{J}_C$ to a sheaf of positive rank. Comparing with the formulas of Proposition \ref{prop2.9}, we find again that $\mathfrak{F}^{\widetilde{H}}$ must act as $1$ on the transcendental part of $\widetilde{H}(\overline{J}_C, \BC)$.
\end{proof}

\subsection{Solving Question \ref{Q1.3}}\label{Sect2.7}

In this section we prove Proposition \ref{prop2.1}, which solves Question~\ref{Q1.3}. We first note some lemmas.

We consider the class of the symplectic form $\sigma \in H^2(\overline{J}_C, \BC)$. For any $\mu \in \mathrm{HT}^2(\overline{J}_C)$, the associated element $e'_\mu(\sigma)$ defines naturally an element in the extended Mukai lattice. This induces an embedding of vector spaces 
\[
\iota: \mathrm{HT}^2(\overline{J}_C) \hookrightarrow \widetilde{H}(\overline{J}_C, \BC), \quad \mu \mapsto e'_\mu(\sigma)
\]
whose image is $\BC \alpha \oplus H^{1,1}(\overline{J}_C)  \oplus \BC \beta$; see \cite[Equation (6.10)]{Markman2}. The identification
\begin{equation}\label{iden}
\mathrm{HT}^2(\overline{J}_C) = \BC \alpha \oplus H^{1,1}(\overline{J}_C)  \oplus \BC \beta
\end{equation}
induced by $\iota$ allows us to view elements on the right-hand side of (\ref{iden}) as elements in $\mathrm{HT}^2(\overline{J}_C)$; this further allows us to write 
\[
e'_\eta \in \mathfrak{g}(\overline{J}_C), \quad \eta \in \BC \alpha \oplus H^{1,1}(\overline{J}_C)  \oplus \BC \beta.
\]
The $\mathfrak{sl}_2$-triple of Example \ref{example} can now be expressed in terms of $\alpha$.

\begin{lem}[{\cite[Lemma 2.9]{Taelman}}]\label{lem2.10}
We have 
\[
e'_\alpha = f_{\sigma} \in \mathfrak{g}(\overline{J}_C).
\]
\end{lem}

We thus obtain an $\mathfrak{sl}_2$-triple
\begin{equation} \label{eq:twosl2}
(e'_\alpha, -h_{\mathrm{hol}}, f'_{\alpha}) = (f_\sigma, -h_{\mathrm{hol}}, -e_{\sigma})\subset \mathfrak{g}(\overline{J}_C).
\end{equation}

Next, since the action of the Fourier transform \eqref{isom} on the extended Mukai lattice is induced by a Hodge isometry, it preserves (\ref{iden}). We have the following comparison.


\begin{lem}\label{lem2.11}
Let $c_1$ be the constant in Proposition \ref{fixsigma}. For any $\eta$ in (\ref{iden}), we have 
\[
\mathfrak{F}^{\mathfrak{g}}(e'_\eta) = c_1 e'_{\mathfrak{F}^{\widetilde{H}}(\eta)} \in \mathfrak{g}(\overline{J}_C).\]
\end{lem}

\begin{proof}
This follows from \cite[Lemma 6.5]{Markman2} together with Proposition \ref{fixsigma}.
\end{proof}

Finally, we recall an important involution on $H^*(\overline{J}_C, \BC)$ which lies in the LLV algebra. The symplectic form $\sigma$ induces an operator 
\[
\Omega: H^*(\overline{J}_C, \BC) \to H^*(\overline{J}_C, \BC)
\]
which acts on the Hodge components $H^{i,j}(\overline{J}_C)$ as Serre duality
\[
\Omega^{i,j}: H^{i,j}(\overline{J}_C) \xrightarrow{\simeq} H^{2g-i,j}(\overline{J}_C).
\]
This operator relates the two types of $\mathfrak{sl}_2$-triples considered by Looijenga--Lunts--Verbitsky \cite{LL, Ver95} and Taelman \cite{Taelman} by conjugation. We collect some useful facts from \cite[Section 6.1]{Markman2} and \cite{Ver99}.

\begin{lem}\label{lem2.12}
We have the following properties of $\Omega$.
\begin{enumerate}
    \item[(a)] The operator 
    \[
    \Omega \in \mathrm{End}\left( H^*(\overline{J}_C, \BC) \right)
    \]
    is an involution lying in $\mathfrak{g}(\overline{J}_C)$ which conjugates $h,h'$, \emph{i.e.},
    \[
    \Omega \circ h \circ \Omega^{-1} = h'.
    \]
    \item[(b)] For any element $\mu \in \mathrm{HT}^2(\overline{J}_C)$, which yields an element $\mu\in H^2(\overline{J}_C, \BC)$ via the natural identification $\mathrm{HT}^2(\overline{J}_C)= H^2(\overline{J}_C, \BC)$ induced by $\sigma$, the operator $\Omega$ conjugates $e_{\mu}$ and~$e'_\mu$, \emph{i.e.},
    \[
    \Omega \circ e_\mu \circ \Omega^{-1} = e'_{\mu}.
    \]
    \item[(c)] Under the notation of (b), for any $\eta \in H^{1,1}(\overline{J}_C)$ we have
    \[
    \Omega\circ e_{\eta}\circ \Omega^{-1} = [f_{\sigma}, e_{\eta}] \in \mathfrak{g}(\overline{J}_C).
    \]
\end{enumerate}
\end{lem}

\begin{proof}
    Parts (a, b) are in \cite[Section 6.1]{Markman2}, and (c) is given by \cite[Proposition 9.7]{Ver99}.
\end{proof}

\begin{cor}\label{cor2.13}
For any $\eta \in H^{1,1}(\overline{J}_C)$, we have
\[
e_\eta = [e_\sigma, e'_\eta] \in \mathfrak{g}(\overline{J}_C).
\]
\end{cor}

\begin{proof}
By Lemma \ref{lem2.12}(b, c), it suffices to show that
\begin{equation}\label{iden2}
e_\eta = [e_\sigma, [f_\sigma, e_\eta]].
\end{equation}
It suffices to treat the case where we have three classes $\eta_1, \eta_2, \eta_3$ satisfying \eqref{eq:condition2.4} with
\[
\eta = \eta_1,\quad e_{\sigma} = e_{\sigma_{23}},\quad f_{\sigma} = f_{\sigma_{23}}.
\]
Under this assumption, we calculate the right-hand side of (\ref{iden2}):
\begin{align*}
    [e_\sigma, [f_\sigma, e_\eta]] &= \left[  \frac{1}{2}(e_{\eta_2} +\mathfrak{i}e_{\eta_3}), \frac{1}{2}(-K_{12} + \mathfrak{i} K_{13}) \right]  \\
    & = \frac{1}{4} ( -[e_{\eta_2}, K_{12}] - [e_{\eta_3}, K_{13}])\\
    & = \frac{1}{4}(2e_{\eta_1}+2e_{\eta_1}) = e_{\eta},
\end{align*}
where we used the Verbitsky relations in Proposition \ref{prop2.3}.
\end{proof}

\begin{proof}[Proof of Proposition \ref{prop2.1}]
Now we are ready to prove Proposition \ref{prop2.1}. We first introduce some notation for convenience. We set 
\[
\overline{\Theta} : = \mathfrak{F}^{\widetilde{H}}(\alpha) \in \widetilde{H}(\overline{J}_C, \BC).
\]
We will use ``$\mathrm{cst}$'' to denote a constant which is not important for us. So Proposition \ref{prop2.9} reads
\begin{equation}\label{EML1}
\overline{\Theta} = -c_0\Theta + \mathrm{cst}\cdot \beta, \quad \mathfrak{F}^{\widetilde{H}}(\overline{\Theta}) = -\alpha + \mathrm{cst}\cdot \sH.
\end{equation}

Since all the calculations in this section only concern the four classes $\alpha, \beta, \Theta, \sH$, we may use the conjugation of Lemma \ref{lem2.12} to reduce them to the Lie algebra model of Section \ref{sec2.4} with the four classes $\eta_1, \eta_2, \eta_3, \eta_4$. More precisely, we may assume that after applying the conjugation 
\begin{equation}\label{conju}
\Omega \circ (-) \circ \Omega^{-1}: \mathfrak{g}(\overline{J}_C) \rightarrow \mathfrak{g}(\overline{J}_C),
\end{equation}
\begin{enumerate}
    \item[(i)] $(e'_{\alpha}, f'_{\alpha}), -e'_\beta$ become $(e_{\sigma_{12}}, f_{\sigma_{12}}), e_{\overline{\sigma}_{12}}$ respectively;
    \item[(ii)] $e'_{\overline{\Theta}}, -c_0e'_{\sH}$ become $e_{\sigma_{34}}, e_{\overline{\sigma}_{34}}$ respectively.
    \end{enumerate}
Furthermore, we use (\ref{conju}) to define the operators $f'_\beta, f'_{\overline{\Theta}}, f'_{\sH}$, completing the $\mathfrak{sl}_2$-triples with the raising operators $e'_\beta, e'_{\overline{\Theta}}, e'_{\sH}$ respectively. For example
\[
f'_{\overline{\Theta}} := \Omega \circ f_{\sigma_{34}} \circ \Omega^{-1}, \quad \mathrm{etc}.
\]

Recall the operators (\ref{operators}). Combining Corollary \ref{cor2.13} with \eqref{eq:twosl2} and (\ref{EML1}), we have 
      \[
      e_0= [e_\sigma, e'_{\Theta}] = -[f'_\alpha, e'_\Theta] = -[f'_\alpha, -c_0e'_{\overline{\Theta}}]+ \mathrm{cst}\cdot [f'_{\alpha}, e'_\beta].
    \] 
By Proposition \ref{prop2.4}(c) and (\ref{conju}), we obtain immediately that $[f'_{\alpha}, e'_\beta] =0$. Therefore
\begin{equation}\label{e_0}
e_0= c_0[f'_\alpha, e'_{\overline{\Theta}}].
\end{equation}
Since the action of $\mathfrak{F}^{\mathfrak{g}}$ on $\mathfrak{g}(\overline{J}_C)$ is linear on the $e'$-operators by Lemma \ref{lem2.11} and it preserves~$h'$, by passing through (\ref{conju}) this action is also linear on the $e$-operators and it preserves $h$; see Lemma~\ref{lem2.12}(a, b). Therefore it also acts linearly on the $f$-operators:
\[
 f_{\sigma_{12}} \mapsto f_{\sigma_{34}}, \quad f_{\sigma_{34}} \mapsto -f_{\sigma_{12}} + \mathrm{cst}\cdot f_{\overline{\sigma}_{34}}.
\]
In particular, we deduce that
\[
\mathfrak{F}^{\mathfrak{g}}(f'_\alpha) = c_1f'_{\overline{\Theta}}, \quad \mathfrak{F}^{\mathfrak{g}}(f'_{\overline{\Theta}}) = -c_1f'_\alpha + \mathrm{cst}\cdot f'_{\sH}.
\]
Here we used (\ref{EML1}) in the equations above.

Now we obtain from (\ref{e_0}) that
\begin{align*}
    f_0= -\mathfrak{F}^\mathfrak{g}(e_0) &=-c_0[\mathfrak{F}^\mathfrak{g}(f'_\alpha), \mathfrak{F}^\mathfrak{g}(e'_{\overline{\Theta}})] \\
    &= -c_0[c_1f'_{\overline{\Theta}}, -c_1e'_\alpha]+ \mathrm{cst}\cdot [f'_{\overline{\Theta}}, e'_{\sH}] \\
    & = c_0[f'_{\overline{\Theta}},  e'_\alpha].
\end{align*}
Here the vanishing $[f'_{\overline{\Theta}}, e'_{\sH}]=0$ follows again from Proposition \ref{prop2.4}(c) via (\ref{conju}). At this point, we have already shown that $f_0$ is of degree $-2$, since $f'_{\overline{\Theta}}, e'_\alpha$ are of degrees $0$, $-2$ respectively; more precisely
\[
f_0: H^{i,j}(\overline{J}_C) \to H^{i-1,j-1}(\overline{J}_C).
\]

Finally, the operators $(e_0, h_0=[e_0,f_0], f_0)$ form an $\mathfrak{sl}_2$-triple by Proposition \ref{prop2.5}.
\end{proof}

\subsection{Solving Question \ref{Q1.4}}
We solve Question \ref{Q1.4} and complete the proof of Theorems \ref{thm0.2} and \ref{thm0.4}.

We first note that the $\mathfrak{sl}_2$-triple $(e_0, h_0, f_0)$ is Fourier-conjugate.

\begin{prop}
    We have
    \[
\mathfrak{F}^\mathfrak{g}(e_0)=-f_0,\quad \mathfrak{F}^\mathfrak{g}(f_0)=-e_0,\quad \mathfrak{F}^\mathfrak{g}(h_0)=-h_0.
\]
\end{prop}

\begin{proof}
The first equation follows from the definition of $f_0$. The second equation is obtained similarly as in the proof of Proposition \ref{prop2.1}:
\begin{align*}
    \mathfrak{F}^\mathfrak{g}(f_0) &=c_0[\mathfrak{F}^\mathfrak{g}(f'_{\overline{\Theta}}), \mathfrak{F}^\mathfrak{g}(e'_{\alpha})] \\
    &= c_0[-c_1f'_{\alpha}, c_1e'_{\overline{\Theta}}]+ \mathrm{cst}\cdot [f'_{\alpha}, e'_\beta] \\
    & = -c_0[f'_{\alpha},  e'_{\overline{\Theta}}] = -e_0.
\end{align*}
The last equation follows from the first two.
\end{proof}

\begin{rmk}
For a general compactified Jacobian fibration, the Fourier transform $\mathfrak{F}$ is \emph{not} an involution on the cohomology $H^*(\overline{J}_C, \BC)$.\footnote{It is not an involution even for the Beauville--Mukai system.} Therefore, even if the operators $(e_0, h_0, f_0)$ defined in (\ref{operators}) form an $\mathfrak{sl}_2$-triple, we do not seem to get for free that
\[
e_0 = -\mathfrak{F}^{-1}\circ f_0 \circ \mathfrak{F}.
\]
In particular, it is unclear if the induced decomposition is Fourier-stable. 
\end{rmk}

We have concluded that the decomposition induced by $(e_0, h_0, f_0)$ is Fourier-stable. We normalize its indices to make the fundamental class $[\overline{J}_C] \in H^0(\overline{J}_C, \BC)$ lie in $H_{(0)}^0(\overline{J}_C, \BC)$. Since all three operators $e_0, h_0, f_0$ are rational, this gives rise to a decomposition with rational coefficients: 
\begin{equation*}\label{eqn223}
H^*(\overline{J}_C, \BQ) = \bigoplus_{i} H_{(i)}^*(\overline{J}_C, \BQ), \quad \mathfrak{F}\left( H_{(i)}^*(\overline{J}_C, \BQ) \right) = H_{(2g-i)}^{*}(\overline{J}_C, \BQ).
\end{equation*}

Finally, we recall the following theorem concerning a canonical $\mathfrak{sl}_2$-triple associated with a general Lagrangian fibration $\pi: X\to B$. It summarizes several results from \cite{SY}.

\begin{thm}[Canonical $\mathfrak{sl}_2$-triple \cite{SY}]\label{thm2.16}
Let $\pi: X \to B$ be a Lagrangian fibration from a compact hyper-K\"ahler variety. Let $\Theta$ be a relatively ample class satisfying $q(\Theta)=0$ with respect to the BBF form $q(-)$. Then there is a unique $\mathfrak{sl}_2$-triple
\[
(e_0, h_0, f_0) \subset \mathfrak{g}(X) \subset \mathrm{End}( H^*(X, \BQ) )
\]
with 
\[
e_0= \Theta \cup -, \quad \mathrm{deg}(f_0)=-2, \quad \mathrm{deg}(h_0)=0. 
\]
Moreover, the induced decomposition splits the perverse filtration associated with $\pi: X \to B$, and is multiplicative with respect to the cup-product.
\end{thm}

\begin{proof}
As before, we let $\sH$ be the pullback of a hyperplane class of the base $B$. By the proof of \cite[Theorem 3.1]{SY}, there is a Lie subalgebra
\[
\mathfrak{sl}_2\times \mathfrak{sl}_2 \subset \mathfrak{g}(X),
\]
whose raising operators are given by $e_0$ and $e_1:= \sH\cup -$ respectively, and the decomposition given by the first $\mathfrak{sl}_2$-triple $(e_0,h_0,f_0)$ splits the perverse filtration. Furthermore, this decomposition is multiplicative with respect to the cup-product by \cite[Proposition A.2]{SY}. This proves the existence part.

It remains to prove uniqueness. Assume that $(e_0,h_0,f_0)$ is an $\mathfrak{sl}_2$-triple with the desired degrees. On one hand, the (canonical) Jacobson--Morosov filtration associated with $e_0$ splits the (also canonical) perverse filtration on each $H^i(X, \BQ)$ up to an index shift uniquely determined by the cohomology degree $i$. On the other hand, this splitting can be read off from the decomposition induced by $(e_0, h_0, f_0)$ as the latter respects cohomology degrees. Therefore, the grading operator $h_0$ is uniquely determined, which further determines $f_0$.
\end{proof}

In our setting $\pi: \overline{J}_C\to B$, since we have already proven that the operators (\ref{operators}) form an~$\mathfrak{sl}_2$-triple with the desired degrees, this has to coincide with the canonical $\mathfrak{sl}_2$-triple of Theorem~\ref{thm2.16}. This completes the proofs of Theorems \ref{thm0.2} and \ref{thm0.4}. \qed

\begin{rmk}\label{rmk2.16}
    In the setting of the Beauville--Mukai system, we show that we can construct~$f_0$ of Theorem \ref{thm2.16} from Arinkin's sheaf \cite{A1, A2}. Therefore all three operators of this $\mathfrak{sl}_2$-triple are given by relative correspondences over $B$, induced by algebraic cycles on $\overline{J}_C\times_B \overline{J}_C$. This is not obvious at all from the perspective of \cite{SY}. We refer to Section \ref{Sec3} for further discussions on a possible motivic lifting of this $\mathfrak{sl}_2$-triple.
\end{rmk}

\begin{rmk}
    The multiplicativity of the \emph{perverse filtration}
    \[
    P_kH^*(\overline{J}_C, \BQ) = \bigoplus_{i\leq k} H^*_{(i)}(\overline{J}_C, \BQ)
    \]
    for any compactified Jacobian fibration $\pi: \overline{J}_C\to B$ as in Section \ref{Sec1.1} has been proven in~\cite{MSY} using the convolution product. However, it is not clear if the techniques of \cite{MSY} can upgrade the multiplicativity from the filtration level to the decomposition even in the Lagrangian setting. Here the multiplicativity of the decomposition is for reasons completely different from~\cite{MSY}. It relies heavily on compact hyper-K\"ahler geometry and is deduced essentially from the multiplicativity of the \emph{Hodge decomposition}. 
\end{rmk}

\section{Motivic lifting of the generalized Beauville decomposition}\label{Sec3}

In this section we propose a lifting of the generalized Beauville decomposition and the $\mathfrak{sl}_2$-triple to the level of Chow groups/motives for certain compactified Jacobian fibrations. As discussed in Section \ref{Sec1.4}, evidence suggests that the \emph{Lagrangian} condition should be a natural setup for the generalized Beauville decomposition and the $\mathfrak{sl}_2$-triple to happen, in which case we expect both to happen \emph{motivically}. Explicit computations are carried out to verify the proposal in the case of elliptic $K3$ surfaces. At the end of this section, we will also discuss relations with the Beauville--Voisin conjectures on the Chow ring of compact hyper-K\"ahler varieties.

\subsection{Motivic Beauville decomposition and $\mathfrak{sl}_2$-triple}

A natural framework for the motivic lifting is the theory of relative Chow motives of Corti--Hanamura \cite{CH}. It is built to be compatible with the decomposition theorem, adapts well to non-proper bases (\emph{e.g.}~the Hitchin fibration), and admits natural Chow/homological realizations \cite{GHM}. We refer to \cite[Chapter~8]{MNP} and \cite[Section 2.2]{MSY} for a brief review of the theory.

Recall that over a nonsingular base $B$, the group of degree $k$ relative correspondences between two proper morphisms $X \to B$, $Y \to B$ with $X, Y$ nonsingular is
\[
\mathrm{Corr}^k_B(X, Y) := \mathrm{CH}_{\dim Y - k}(X \times_B Y, \BQ).
\]
Compositions of relative correspondences are defined via refined intersection theory. The category of relative Chow motives $\mathrm{CHM}(B)$ consists of objects triples $(X, \mathfrak{p}, m)$ where $X \to B$ is a proper morphism with $X$ nonsingular, $\mathfrak{p} \in \Corr^0_B(X, X)$ is a projector, and $m \in \mathbb{Z}$. In particular, the motive of $X$ is given by $h(X/B) := (X, [\Delta_{X/B}], 0)$ where $\Delta_{X/B}$ is the relative diagonal. Morphisms between two motives $M = (X, \mathfrak{p}, m)$, $N = (Y, \mathfrak{q}, n)$ are
\[
\Hom_{\mathrm{CHM}(B)}(M, N) := \mathfrak{q} \circ \Corr^{n - m}_B(X, Y) \circ \mathfrak{p}.
\]
Finally, the degree $k$ Chow group of $M = (X, \mathfrak{p}, m)$ is defined by
\[
\mathrm{CH}^k(M, \BQ) := \Hom_{\mathrm{CHM}(B)}(h(B/B), M(k))
\]
where $M(k) := (X, \mathfrak{p}, m + k)$ is the $k$-th Tate twist of $M$.

We propose the following motivic versions of the generalized
Beauville decomposition and the $\mathfrak{sl}_2$-action for compactified Jacobian fibrations under the \emph{Lagrangian} assumption.

\begin{conj} \label{conj3.1}
Let $\pi: \overline{J}_C \to B$ be the compactified Jacobian fibration associated with a flat family of integral projective locally planar curves of arithmetic genus $g$ over a nonsingular variety $B$. We assume that $\overline{J}_C$ is nonsingular, and that the fibration $\pi$ is Lagrangian with respect to a holomorphic symplectic form on $\overline{J}_C$. Then there exists a decomposition
\begin{equation} \label{eq:motdec}
h(\overline{J}_C/B) = \bigoplus_{i = 0}^{2g} h_i(\overline{J}_C/B) \in \mathrm{CHM}(B), \quad h_i(\overline{J}_C/B) = (\overline{J}_C, \mathfrak{p}_i, 0)
\end{equation}
whose homological realization splits the perverse filtration on $R\pi_*\mathbb{Q}_{\overline{J}_C}$, and which is
\begin{enumerate}
\item[(a)] stable under the Fourier transform
\[
\mathfrak{p}_j \circ \FF \circ \mathfrak{p}_i = 0, \quad i + j \neq 2g;
\]
\item[(b)] multiplicative with respect to the cup-product
\[
\mathfrak{p}_k \circ [\Delta_{\overline{J}_C/B}^{\mathrm{sm}}] \circ (\mathfrak{p}_i \times \mathfrak{p}_j) = 0, \quad i + j \neq k.
\]
\end{enumerate}
\end{conj}

Here $\Delta_{\overline{J}_C/B}^{\mathrm{sm}}$ is the small relative diagonal of $\overline{J}_C$ in $\overline{J}_C \times_B \overline{J}_C \times_B \overline{J}_C$ and we refer to \cite[Section 2.2.3]{MSY} for the multiplicative structure of $\mathrm{CHM}(B)$.

\begin{conj} \label{conj3.2}
Let $\pi: \overline{J}_C \to B$ be as in Conjecture \ref{conj3.1}. Then there is an $\mathfrak{sl}_2$-triple
\[
e_0 \in \Corr^1_B(\overline{J}_C, \overline{J}_C), \quad f_0 \in \Corr^{-1}_B(\overline{J}_C, \overline{J}_C), \quad h_0 \in \Corr^0_B(\overline{J}_C, \overline{J}_C)
\]
which induces the Fourier-stable multiplicative decomposition of motives \eqref{eq:motdec}
\begin{equation} \label{eq:eigen}
h_0 \circ \mathfrak{p_i} = (i - g) \mathfrak{p_i}, \quad 0 \leq i \leq 2g.
\end{equation}
\end{conj}

Note that the identities \eqref{eq:eigen} together with $[\Delta_{\overline{J}_C/B}] = \sum_{i = 0}^{2g} \mathfrak{p}_i$ already imply that the $\mathfrak{p}_i$ are orthogonal projectors. Indeed, we have
\[
\mathfrak{p}_i =  [\Delta_{\overline{J}_C/B}] \circ \mathfrak{p}_i = \sum_{j = 0}^{2g} \mathfrak{p}_j \circ \mathfrak{p}_i.
\]
Comparing eigenvalues with respect to $h_0$ on both sides, we see that $\mathfrak{p}_j \circ \mathfrak{p}_i = \delta_{ij}$. We also deduce
\[
\mathfrak{p_j} \circ h_0 = (j - g) \mathfrak{p_j}, \quad 0 \leq j \leq 2g
\]
since $h_0 = \sum_{i = 0}^{2g} h_0 \circ \mathfrak{p}_i = \sum_{i = 0}^{2g} (i - g) \mathfrak{p}_i$. Then, Fourier-stability can be detected by verifying the identity
\[
\FF^{-1} \circ h_0 \circ \FF = -h_0,
\]
and multiplicativity by
\begin{multline} \label{eq:motmult}
[\Delta_{\overline{J}_C/B}^{\mathrm{sm}}] \circ \left((h_0 + g \cdot [\Delta_{\overline{J}_C/B}]) \times [\Delta_{\overline{J}_C/B}] + [\Delta_{\overline{J}_C/B}] \times (h_0 + g \cdot [\Delta_{\overline{J}_C/B}])\right) \\
= (h_0 + g \cdot [\Delta_{\overline{J}_C/B}]) \circ [\Delta_{\overline{J}_C/B}^{\mathrm{sm}}]. 
\end{multline}
Similar ideas involving $h_0$ have already been employed in \cite{NOY}.

We make two more comments on the conjectures above. First, when $B$ is proper there is a natural pushforward functor
\begin{equation} \label{eq:relabs}
\mathrm{CHM}(B) \to \mathrm{CHM}(\mathrm{pt})
\end{equation}
from the category of relative Chow motives to the category of \emph{absolute} Chow motives. It is defined by pushing forward
\[
\Corr^0_B(X, X) = \mathrm{CH}_{\dim X}(X \times_B X, \BQ) \to \mathrm{CH}_{\dim X}(X \times X, \BQ) = \Corr^0(X, X)
\]
via the closed embedding $X \times_B X \hookrightarrow X \times X$. In this case we can formulate weaker versions of Conjectures \ref{conj3.1} and \ref{conj3.2} concerning the absolute Chow motive $h(\overline{J}_C) \in \mathrm{CHM}(\mathrm{pt})$, \emph{i.e.}, the image of~$h(\overline{J}_C/B) \in \mathrm{CHM}(B)$ under \eqref{eq:relabs}. These weaker versions in turn specialize to the cohomological statements of Theorems \ref{thm0.2} and \ref{thm0.4} for the Beauville--Mukai system. We also mention here the recent work of Ancona--Cavicchi--Laterveer--Sacc\`a \cite{ACLS} on decomposing the relative/absolute \emph{homological} motive of a compact hyper-K\"ahler variety carrying a Lagrangian fibration. In particular, using their results one can give another proof that the generalized Beauville decomposition of Theorem \ref{thm0.4} (which coincides with the decomposition in \cite{SY} by Theorem \ref{thm2.16}) is induced by relative correspondences over $B$; see \cite[Proposition 5.3 and Corollary 6.5]{ACLS}.\footnote{The arguments of \cite{ACLS} work equally for twisted Beauville--Mukai systems associated with an irreducible curve class and for (twisted) LSV fibrations \cite{LSV}.}

\begin{rmk}
There is a subtle difference between an $\mathfrak{sl}_2$-action on the relative Chow motive~$h(\overline{J}_C/B)$ and an $\mathfrak{sl}_2$-action on the absolute Chow motive $h(\overline{J}_C)$ induced by relative correspondences over $B$. The former is significantly stronger as it requires the $\mathfrak{sl}_2$-relations to hold in the Chow groups $\mathrm{CH}_*(\overline{J}_C \times_B \overline{J}_C, \BQ)$. Meanwhile, we also lose the realization
\[
\mathrm{CHM}(B) \to D^b_c(B)
\]
when passing from $\mathrm{CHM}(B)$ to $\mathrm{CHM}(\mathrm{pt})$.
\end{rmk}

Second, we believe that the motivic $\mathfrak{sl}_2$-triple can be constructed following the strategy of Section~\ref{sec1.3}. To begin with, the generalized theta divisor should be upgraded to the Chow level, \emph{i.e.}, $\Theta \in \mathrm{CH}^1(h_2(\overline{J}_C/B), \BQ)$ such that
\begin{equation} \label{eq:genchow}
\Theta^{g + 1} = 0 \in \mathrm{CH}^{g + 1}(\overline{J}_C, \BQ).
\end{equation}
When the base $B$ is proper the existence of $\Theta$ is well-studied in the context of compact hyper-K\"ahler geometry and is part of the Beauville--Voisin philosophy (see Section \ref{Sec3.3} below). Notably Rie{\ss} \cite{R} proved that assuming the hyper-K\"ahler SYZ conjecture, any $\BQ$-divisor $D$ with $q(D) = 0$ on a compact hyper-K\"ahler variety $X$ of dimension $2n$ satisfies
\[
D^{n + 1} = 0 \in \mathrm{CH}^{n + 1}(X, \BQ).
\]
Next, with the divisor $\Theta$ in \eqref{eq:genchow} we define
\begin{equation} \label{eq:defe0}
e_0 := \Delta_{\overline{J}_C/B*}\Theta \in \Corr^1_B(\overline{J}_C, \overline{J}_C)
\end{equation}
where $\Delta_{\overline{J}_C/B} : \overline{J}_C \hookrightarrow \overline{J}_C \times_B \overline{J}_C$ is the relative diagonal map, and
\[
f_0 := -\FF^{-1} \circ e_0 \circ \FF, \quad h_0 := [e_0, f_0]
\]
following Section \ref{sec1.3}. We expect $f_0$ to lie in $\Corr^{-1}_B(\overline{J}_C, \overline{J}_C)$ (hence $h_0 \in \Corr^0_B(\overline{J}_C, \overline{J}_C)$), and that $(e_0, h_0, f_0)$ form the desired $\mathfrak{sl}_2$-triple of Conjecture \ref{conj3.2}.

We now illustrate how this proposal works for elliptic $K3$ surfaces.

\subsection{Elliptic $K3$ surfaces} \label{sec3.2}

In this section, let $\pi: S \to \BP^1$ be an elliptic fibration from a projective $K3$ surface with only integral fibers and a section $s: \BP^1 \to S$, so that $S$ is isomorphic to a genus $1$ compactified Jacobian family $\overline{J}_C$. To set up the notation, let
\[
\mathsf{s} := [s(\BP^1)] \in \mathrm{CH}^1(S, \BQ), \quad \mathsf{f} := [\pi^{-1}(\mathrm{pt})] \in \mathrm{CH}^1(S, \BQ)
\]
be the classes of the section and fibers. A generalized theta divisor is then given by
\[ \Theta := \mathsf{s} + \mathsf{f}, \quad \Theta^2 = 0 \in \mathrm{CH}^2(S, \BQ).\]
Also let $c \in \mathrm{CH}^2(S, \BQ)$ be the Beauville--Voisin distinguished class \cite{BV} which is supported on any rational curve of $S$.

We define self-correspondences in $\Corr^{0}_{\BP^1}(S, S)$:
\[
\mathfrak{p}_0 := p_1^*\Theta, \quad \mathfrak{p}_2 := p_2^*\Theta, \quad \mathfrak{p}_1 := [\Delta_{S/\mathbb{P}^1}] - \mathfrak{p}_0 - \mathfrak{p}_2,
\]
where $p_1, p_2: S \times_{\BP^1} S \to S$ are the two projections. It is straightforward to check that the $\mathfrak{p}_i$ are orthogonal projectors, and the resulting motivic decomposition
\begin{equation} \label{eq:motdeck3}
h(S/\BP^1) = \bigoplus_{i = 0}^2 h_i(S/\BP^1) \in \mathrm{CHM}(\BP^1), \quad h_i(S/\BP^1) = (S, \mathfrak{p}_i, 0)
\end{equation}
specializes to a splitting of the perverse filtration on $R\pi_*\BQ_S$.

Concerning the $\mathfrak{sl}_2$-triple we set as in \eqref{eq:defe0}
\[
e_0 := \Delta_{S/\BP^1*} \Theta \in \Corr^1_{\BP^1}(S, S),
\]
and for dimension reasons
\[
f_0 := [S \times_{\BP^1} S] \in \Corr^{-1}_{\BP^1}(S, S).
\]
Then we have
\[
h_0 := [e_0, f_0] = p_2^*\Theta - p_1^*\Theta \in \Corr^0_{\BP^1}(S, S).
\]

We give a list of desired properties of the motivic decomposition and the $\mathfrak{sl}_2$-triple in the following proposition, verifying Conjectures \ref{conj3.1} and \ref{conj3.2} for $\pi: S \to \BP^1$.

\begin{prop}
We have
\begin{enumerate}
\item[(a)] $[h_0, e_0] = 2e_0$, $[h_0, f_0] = -2f_0$;
\item[(b)] $h_0 \circ \mathfrak{p_i} = (i - 1)\mathfrak{p_i}$ for $i = 0, 1, 2$;
\item[(c)] $\FF^{-1} \circ e_0 \circ \FF = -f_0$, $\FF^{-1} \circ f_0 \circ \FF = -e_0$, $\FF^{-1} \circ h_0 \circ \FF = -h_0$;
\item[(d)] $[\Delta^{\mathrm{sm}}_{S/\mathbb{P}^1}] \circ (h_0 \times [\Delta_{S/\mathbb{P}^1}] + [\Delta_{S/\mathbb{P}^1}] \times h_0 + [\Delta_{S/\mathbb{P}^1}] \times [\Delta_{S/\mathbb{P}^1}]) = h_0 \circ [\Delta^{\mathrm{sm}}_{S/\mathbb{P}^1}]$.
\end{enumerate}
\end{prop}

Here (b) shows that the motivic decomposition \eqref{eq:motdeck3} is induced by the $\mathfrak{sl}_2$-triple $(e_0, h_0, f_0)$ from (a), (c) shows that it is Fourier-stable, and (d) which is a reformulation of \eqref{eq:motmult} verifies the multiplicativity of the motivic decomposition.

\begin{proof}
For (a) we observe that
\begin{equation} \label{eq:thetadiag}
e_0 = \Delta_{S/\BP^1*}\Theta = p_1^*\Theta \,p_2^*\Theta \in \mathrm{CH}_1(S \times_{\BP^1} S, \BQ).
\end{equation}
In fact we have $\Delta_{S/\BP^1*}\mathsf{s} = p_1^*\mathsf{s} \, p_2^*\mathsf{s}$ by definition. For $\Delta_{S/\BP^1*}\mathsf{f}$ we may take $\mathsf{f}$ to be the class of a (singular) rational fiber of $\pi$. Since $[\Delta_{\BP^1}] = [\mathrm{pt} \times \BP^1] + [\BP^1 \times \mathrm{pt}]$, we find
\begin{equation} \label{eq:s+f}
\Delta_{S/\BP^1*}\mathsf{f} = p_1^*\mathsf{s} \, p_2^*\mathsf{f} + p_1^*\mathsf{f} \, p_2^*\mathsf{s} \in \mathrm{CH}_1(S \times_{\BP^1} S, \BQ).
\end{equation}
Altogether we have
\[
\Delta_{S/\BP^1*}\Theta = \Delta_{S/\BP^1*}(\mathsf{s} + \mathsf{f}) = p_1^*\mathsf{s} \, p_2^*\mathsf{s} + p_1^*\mathsf{s} \, p_2^*\mathsf{f} + p_1^*\mathsf{f} \, p_2^*\mathsf{s} = p_1^*(\mathsf{s} + \mathsf{f}) p_2^*(\mathsf{s} + \mathsf{f}) = p_1^*\Theta \,p_2^*\Theta
\]
since $p_1^*\mathsf{f} \, p_2^*\mathsf{f} = 0$.

We compute
\begin{align*}
[h_0, e_0] & = h_0 \circ e_0 - e_0 \circ h_0 \\
& = (p_2^*\Theta - p_1^*\Theta) \circ (p_1^*\Theta \,p_2^*\Theta) - (p_1^*\Theta \,p_2^*\Theta) \circ (p_2^*\Theta - p_1^*\Theta) \\
& = p_1^*\Theta \,p_2^*\Theta - (- p_1^*\Theta \,p_2^*\Theta) = 2e_0,
\end{align*}
where the second to last equality uses $\Theta^2 = 0$. The other identity $[h_0, f_0] = -2f_0$ and part~(b) are both straightforward.

For (c) we will need the following version of Proposition~\ref{prop2.9}:
\begin{gather*}
\mathfrak{F}([S]) = -\Theta + c, \quad 
\mathfrak{F}(c) = \mathsf{f}, \quad \mathfrak{F}(\Theta) = [S] - \mathsf{f}, \quad \mathfrak{F}(\mathsf{f}) = -c, \\
\mathfrak{F}^{-1}([S]) = \Theta + c, \quad 
\mathfrak{F}^{-1}(c) = -\mathsf{f}, \quad \mathfrak{F}^{-1}(\Theta) = -[S] - \mathsf{f}, \quad \mathfrak{F}^{-1}(\mathsf{f}) = c,
\end{gather*}
which is a consequence of the more general result \cite{H, SYZ} that derived equivalences preserve the Beauville--Voisin ring. We also recall from \cite{A2} that $\FF$ is symmetric with respect to the two factors of $S \times_{\BP^1} S$.

We compute
\begin{align*}
\FF^{-1} \circ e_0 \circ \FF & = \FF^{-1} \circ (p_1^*\Theta \,p_2^*\Theta) \circ \FF \\
& = \FF^{-1} \circ (p_1^*\FF(\Theta) \,p_2^*\Theta) \\
& = p_1^*\FF(\Theta) \, p_2^*\FF^{-1}(\Theta) \\
& = p_1^*([S] - \mathsf{f}) \, p_2^*(-[S] - \mathsf{f}) \\
& = -[S \times_{\BP^1} S] = -f_0,
\end{align*}
where the second to last equality uses that $\mathsf{f}$ is pulled back from the base $\BP^1$. Similarly, we~have
\begin{align*}
\FF^{-1} \circ f_0 \circ \FF & = \FF^{-1} \circ [S \times_{\BP^1} S] \circ \FF \\
& = \FF^{-1} \circ p_1^*\FF([S]) \\
& = p_1^*\FF([S]) \, p_2^*\FF^{-1}([S]) \\
& = p_1^*(-\Theta + c) \, p_2^*(\Theta + c) \\
& = -p_1^*\Theta \, p_2^*\Theta = -e_0,
\end{align*}
where the second to last equality is because $-p_1^*\Theta\,p_2^*c + p_1^*c\,p_2^*\Theta$ is a $0$-cycle of degree $0$ supported on a chain of rational curves and hence vanishes. The third identity of (c) follows from the first two.

We now prove (d). By plugging in the definition of $h_0$, we may rewrite (d) as
\begin{multline*}
q_{23}^*\Delta_{S/\BP^1*}\Theta - q_1^*\Theta \, q_{23}^*[\Delta_{S/ \BP^1}] + q_{13}^*\Delta_{S/\BP^1*}\Theta - q_2^*\Theta \, q_{13}^*[\Delta_{S/ \BP^1}] + [\Delta_{S/\BP^1}^{\mathrm{sm}}] \\
{}= q_3^*\Theta \, q_{12}^*[\Delta_{S/ \BP^1}] - q_{12}^*\Delta_{S/\BP^1*}\Theta \in \mathrm{CH}_2(S \times_{\BP^1} S \times_{\BP^1} S, \BQ),
\end{multline*}
where the $q_i: S \times_{\BP^1} S \times_{\BP^1} S \to S$, $q_{ij}: S \times_{\BP^1} S \times_{\BP^1} S \to S \times_{\BP^1} S$ are the natural projections. In a more symmetric form we find
\[
[\Delta_{S/\BP^1}^{\mathrm{sm}}] - \left(q_1^*\Theta \, q_{23}^*[\Delta_{S/ \BP^1}] + \textrm{permutations}\right) + \left(q_{23}^*\Delta_{S/\BP^1*}\Theta + \textrm{permutations}\right) = 0.
\]
Applying \eqref{eq:thetadiag} and \eqref{eq:s+f} we are further reduced to
\begin{equation} \label{eq:relbv}
[\Delta_{S/\BP^1}^{\mathrm{sm}}] - \left(q_1^*\mathsf{s} \, q_{23}^*[\Delta_{S/ \BP^1}] + \textrm{permutations}\right) + (q_2^*\mathsf{s} \, q_3^*\mathsf{s} + \textrm{permutations}) = 0.
\end{equation}

The proof of \eqref{eq:relbv} follows closely the original argument of Beauville--Voisin in \cite{BV}. First, the left-hand side of \eqref{eq:relbv} vanishes when restricted to the smooth locus of $\pi: S \to \BP^1$ (see \emph{e.g.}~\cite{MY} for a proof of this fact). Hence by the localization sequence, it is supported on the triple products of the (singular) rational fibers. On the other hand, all $\mathfrak{S}_3$-invariant $2$-cycles supported on such triple products are proportional to the effective cycle
\begin{equation} \label{eq:eff}
q_1^*c + q_2^*c + q_3^*c \in \mathrm{CH}_2(S \times_{\BP^1} S \times_{\BP^1} S, \BQ).
\end{equation}
We deduce that the left-hand side of \eqref{eq:relbv} is a multiple of \eqref{eq:eff}.

Next, we push the left-hand side of \eqref{eq:relbv} all the way to the absolute triple product~$S \times S \times S$, which sends $[\Delta_{S/\BP^1}^{\mathrm{sm}}]$ to the small diagonal $[\Delta^{\mathrm{sm}}_S] \in \mathrm{CH}_2(S \times S \times S, \BQ)$.
Using~$\mathsf{s}\mathsf{f} = c$ and the identities
\begin{gather*}
[S \times_{\BP^1} S] = p_1'^*\mathsf{f} + p_2'^*\mathsf{f} \in \mathrm{CH}_3(S \times S, \BQ), \\
[S \times_{\BP^1} S \times_{\BP^1} S] = q'^*_2\mathsf{f} \, q'^*_3\mathsf{f} + q'^*_1\mathsf{f} \, q'^*_3\mathsf{f} +q'^*_1\mathsf{f} \, q'^*_2\mathsf{f} \in \mathrm{CH}_4(S \times S \times S, \BQ)
\end{gather*}
where the $p'_i: S \times S \to S$, $q'_i: S \times S \times S \to S$ (as well as the $q'_{ij}$ below) are the natural projections from the absolute products, we find that the pushforward of $q_1^*\mathsf{s} \, q_{23}^*[\Delta_{S/ \BP^1}]$ in~\eqref{eq:relbv}~is
\begin{equation} \label{eq:chaos}
q'^*_1 c \, q'^*_{23}[\Delta_S] + q'^*_1 \mathsf{s} \, q'^*_{23}\Delta_{S*}\mathsf{f} \in \mathrm{CH}_2(S \times S \times S, \BQ).
\end{equation}
Similarly, the pushforward of $q_2^*\mathsf{s} \, q_3^*\mathsf{s}$ is
\begin{equation} \label{eq:chaos2}
q'^*_{2}c \, q'^*_{3}c + q'^*_{1}\mathsf{f} \, q'^*_{2}\mathsf{s} \, q'^*_{3}c + q'^*_{1}\mathsf{f} \, q'^*_{2}c \, q'^*_{3}\mathsf{s} \in \mathrm{CH}_2(S \times S \times S, \BQ).
\end{equation}
Further expanding \eqref{eq:chaos} using the relation (either by \cite{BV} or by representing $\mathsf{f}$ by a rational~fiber)
\[
\Delta_{S*}\mathsf{f} = p_1'^*c \, p_2'^*\mathsf{f} + p_1'^*\mathsf{f} \, p_2'^*c \in \mathrm{CH}_1(S \times S, \BQ)
\]
and summing up the other two permutations of \eqref{eq:chaos} and \eqref{eq:chaos2}, we conclude that the pushforward of the left-hand side of \eqref{eq:relbv} is precisely
\[
[\Delta_{S}^{\mathrm{sm}}] - (q_1'^*c \, q_{23}'^*[\Delta_{S}] + \textrm{permutations}) + (q_2'^*c \, q_3'^*c + \textrm{permutations}),
\]
which vanishes by \cite{BV}.

As the left-hand side of \eqref{eq:relbv} is a multiple of \eqref{eq:eff} whose pushforward to $S \times S \times S$ is clearly nonzero, this multiple must be $0$. This proves the relation \eqref{eq:relbv}.
\end{proof}

\subsection{Relations to Beauville--Voisin} \label{Sec3.3}
The Beauville--Voisin conjectures \cite{Bea, Voi} refer to a series of open problems concerning the Chow ring/motive of compact hyper-K\"ahler varieties and the behavior of their Chern classes. Roughly speaking, the conjectures predict a multiplicative decomposition of the Chow ring which splits the conjectural Bloch--Beilinson filtration, and for which the Chern classes lie in the ``correct'' components. We will not review the history of the conjectures, but only mention a subconjecture recently investigated by Voisin \cite{Voi23}.

\begin{conj}[{\cite[Conjecture 1.5]{Voi23}}] \label{conjvoi}
Let $X$ be a compact hyper-K\"ahler variety of dimension~$2n$. Then for any $\BQ$-divisor $D$ with $q(D) = 0$ and any Chern monomial $c_I \in \mathrm{CH}^{2k}(X, \BQ)$, we have
\[
D^{n - k + 1} c_I = 0 \in \mathrm{CH}^{n + k + 1}(X, \BQ).
\]
\end{conj}

Here a Chern monomial refers to a monomial in the Chern classes $c_2(X), \ldots, c_{2n}(X)$. When~$X$ admits a Lagrangian fibration $\pi: X \to B$ and $D = \Theta$ is relatively ample, we can reinterpret Conjecture \ref{conjvoi} as a statement on the ``Chow-theoretic perversity'' of the Chern classes. This leads to the following prediction, which for convenience we state for compactified Jacobian fibrations under the Lagrangian assumption over possibly non-proper bases.

\begin{conj} \label{conj3.5}
Let $\pi : \overline{J}_C \to B$ be as in Conjecture \ref{conj3.1}. Then for $i \geq 0$, we have
\[
c_{2i}(\overline{J}_C) \in \mathrm{CH}^{2i}(h_{2i}(\overline{J}_C/B), \BQ).
\]    
\end{conj}

\begin{rmk}
Conjecture \ref{conj3.5} is immediate for the elliptic $K3$ surface $\pi: S \to \BP^1$ of Section~\ref{sec3.2} as $c \in \mathrm{CH}^2(h_2(S/\BP^1), \BQ)$. The weaker, cohomological statement $c_{2i}(\overline{J}_C) \in H^{4i}_{(2i)}(\overline{J}_C, \BQ)$ holds for both the Beauville--Mukai system (and more generally for any Lagrangian fibration $\pi: X \to B$ from a compact hyper-K\"ahler variety by the calculations in \cite[Introduction]{Voi23}) and the Hitchin system. In the Hitchin case, this can be seen as a consequence of $P = W$ and the natural splitting of the weight filtration on the Betti side.
\end{rmk}

\begin{prop}
Assume Conjectures \ref{conj3.1} and \ref{conj3.5}. Then for any $\Theta \in \mathrm{CH}^1(h_2(\overline{J}_C/B), \BQ)$ and any Chern monomial $c_I \in \mathrm{CH}^{2k}(\overline{J}_C, \BQ)$, we have
\[
\Theta^{g - k + 1} c_I = 0 \in \mathrm{CH}^{g + k + 1}(\overline{J}_C, \BQ).
\]
\end{prop}

\begin{proof}
By the multiplicativity of the motivic decomposition \eqref{eq:motdec} and Conjecture \ref{conj3.5}, we find
\[
\Theta^{g - k + 1} c_I \in \mathrm{CH}^{g + k + 1}(h_{2(g - k + 1) + 2k}(\overline{J}_C/B), \BQ) = \mathrm{CH}^{g + k + 1}(h_{2g + 2}(\overline{J}_C/B), \BQ) = 0. \qedhere
\]
\end{proof}

It is plausible that with the techniques of \cite{R, ACLS}, one can possibly deduce Conjecture \ref{conjvoi} from the analogues of Conjectures \ref{conj3.1} and \ref{conj3.5} for more general compact hyper-K\"ahler varieties together with the hyper-K\"ahler SYZ conjecture.

We also comment on how Conjectures \ref{conj3.1}, \ref{conj3.2}, and \ref{conj3.5} are related to the full Beauville--Voisin conjecture. Previously, Oberdieck \cite{Obe} initiated an approach to the multiplicative decomposition problem by lifting Lefschetz triples (or more generally the N\'eron--Severi part of the LLV algebra) to the level of Chow groups/motives. This was carried out in full for the Hilbert schemes of $K3$ surfaces and partially in dimension $4$ in the subsequent papers~\cite{NOY, Kre}. When the compact hyper-K\"ahler variety admits a Lagrangian fibration, the relative structure may lead to a \emph{two-step} construction of the eventual multiplicative decomposition. For illustration purposes we restrict ourselves to the Beauville--Mukai system $\pi: \overline{J}_C \to B$.

\medskip
\noindent {\bf Step 1.} Prove Conjectures \ref{conj3.1}, \ref{conj3.2}, and \ref{conj3.5} for $\pi: \overline{J}_C \to B$. It is worth mentioning again that all three conjectures concerning the \emph{relative} Chow motive of $\overline{J}_C$ are stronger than predicted by Beauville--Voisin, and are interesting in their own right.

\medskip
\noindent {\bf Step 2.} Consider the pushforward of the motivic decomposition \eqref{eq:motdec} via \eqref{eq:relabs}. Look for a second $\mathfrak{sl}_2$-triple $(e_1, h_1, f_1)$ with
\[
e_1 := \Delta_*\sH \in \Corr^1(\overline{J}_C, \overline{J}_C)
\]
where $\sH \in \mathrm{CH}^1(\overline{J}_C, \BQ)$ is the pullback of the hyperplane class of $B \simeq \BP^g$, thus lifting \mbox{$\mathfrak{sl}_2 \times \mathfrak{sl}_2 \subset \mathfrak{g}(\overline{J}_C)$} (see \cite{SY}) to an action on the \emph{absolute} Chow motive of $\overline{J}_C$. Further decompose the image of each $h_i(\overline{J}_C/B)$ into eigen-motives with respect to $h_1$. Finally, rearrange the eigen-motives with respect to $h := h_0 + h_1$, and use it to detect multiplicativity as well as the locations of the Chern classes.

\medskip
Similar two-step approaches have succeeded in proving cases of the Lefschetz standard conjecture for compact hyper-K\"ahler varieties in \cite{Voi22, ACLS}.

\section{Proof of Theorem \ref{thm0.40}}\label{Sec4}
In this section we prove Theorem \ref{thm0.40} using tautological relations on the universal Picard stack. For $g\geq 2$, let $\overline{\mathcal{M}}_{g}$ be the moduli stack of stable curves of genus $g$. It is the Deligne--Mumford compactification of the moduli stack of genus $g$ nonsingular projective curves \mbox{$\CM_g \subset \overline{\CM}_{g}$}, and every curve on the boundary $\overline{\CM}_g\setminus \CM_g$ has at worst simple nodes as singular points. More generally we also consider the Deligne--Mumford moduli stack of stable curves with marked points $\CM_{g,n} \subset \overline{\CM}_{g,n}$. 

We denote by $\mg\subset \overline{\mathcal{M}}_g$ the open locus where the curves are integral and have at most one node. Let $\CC_g^{\leq 1}\to \mg$ be the universal curve with $\pi: \overline{J}_g^{\leq 1}\to \mg$ the corresponding compactified Jacobian fibration.  

Our main result is the following. 

\begin{thm}\label{thm4.1}
We have the following statements concerning generalized theta divisors.
\begin{enumerate}
    \item[(a)] For $g\geq 3$, a generalized theta divisor does not exist for $\pi: \overline{J}_g^{\leq 1} \to \mg$.
    \item[(b)] Let $\overline{\CM}^\mathrm{int}_2 \subset \overline{\CM}_2$ be the open locus of integral stable curves of genus $2$ with $\pi^{\mathrm{int}}: \overline{J}^{\mathrm{int}}_2 \to \overline{\CM}^\mathrm{int}_2$ the corresponding compactified Jacobian fibration. Then a generalized theta divisor does not exist for $\pi^{\mathrm{int}}$.
\end{enumerate}
\end{thm}

By the discussions of Section \ref{Sec1.3}, Theorem \ref{thm4.1} proves a version of Theorem \ref{thm0.40} when the base $B$ is a nonsingular Deligne--Mumford stack. We will explain in Section \ref{final} that the same argument actually also proves Theorem \ref{thm0.40} with $B$ a nonsingular quasi-projective variety.

\subsection{The universal Picard stack}

We recall some basic facts about the universal Picard stack; our reference is \cite{BHPSS}.

Let $\mathfrak{M}_g$ be the moduli stack of (not necessarily stable) nodal curves of genus $g$ and let \mbox{$\mathfrak{C}_g\to \mathfrak{M}_g$} be the universal curve. We consider the universal Picard stack $\ppic_g\to\mathfrak{M}_g$ parameterizing total degree $0$ line bundles. The relative Picard stack $\ppic^\rel_g$ is the quotient of $\ppic_g$ by relative inertia $B\BG_m$ over~$\mathfrak{M}_g$; in other words, it is the rigidification of the universal Picard stack with respect to the automorphism $\BG_m$.


We consider the universal line bundle $\CL$ on the universal curve $p: \mathfrak{C}_{\ppic_g} \to \ppic_g$; see~\cite{BHPSS}. This yields the tautological class
\begin{equation*}\label{eq:theta}
    \theta := -\frac{1}{2}p_*(c_1(\CL)^2) \in \Chow^1(\ppic_g, \BQ).
\end{equation*}
Since this class is invariant under twisting $\CL$ by a line bundle from the base, it descends to
\[
\theta \in
\Chow^1(\ppic^\rel_g, \BQ).
\]

For a family of genus $g$ integral projective nodal curves $C\to B$, the relative compactified Jacobian admits a natural morphism to the relative Picard stack
\[
\overline{J}_C\to \ppic^\rel_g.
\]
This is because the universal sheaf on $C\times_B\overline{J}_C$ can be interpreted as an admissible line bundle on the quasi-stable model of the family of nodal curves by \cite{EP}. It is necessary to allow unstable nodal curves and line bundles which are not of multi-degree $0$ for such a map to exist. In particular we have a natural morphism
\[
\overline{J}_g^{\leq 1} \rightarrow \ppic^\rel_g.
\]

\subsection{Universal double ramification cycle relations}
For the relative compactified Jacobian $\pi:\overline{J}_g^{\leq 1}\to\mg$, we have the following commutative diagram
\begin{equation}\label{eq:Jbar}
    \begin{tikzcd}
    J_{g-1, 2} \ar[r, "t"] & J_{g-1, 2}^{-1} \ar[r,"\epsilon"]\ar[dr,"\pi'"'] & \overline{J}_\delta \ar[d]\ar[r,"\iota"] & \overline{J}_g^{\leq 1} \ar[d,"\pi"] & J_g \ar[l,"j"']\ar[d] \\
    & & \mathcal{M}_{g-1,2} \ar[r,"\iota"] & \mg & \mathcal{M}_g\ar[l,"j"']
    \end{tikzcd}
\end{equation}
where $j:\mathcal{M}_g\hookrightarrow\mg $ is the natural open embedding, $\iota:\CM_{g-1,2} \to \mg$ is the morphism gluing the two marked points on the curve, and $J_{g-1,2}^{-1}$ is the relative Jacobian of degree $-1$ line bundles over $\mathcal{M}_{g-1,2}$. The morphism $\epsilon$ is given by $(\widetilde{C},L)\mapsto (C,\nu_*L)$ where $\nu: \widetilde{C}\to C$ is the normalization of a nodal curve with one self-node. The image of $\iota\circ\epsilon$ is the singular locus of~$\pi$. The relative Jacobian $J_{g-1,2}^{-1}$ is isomorphic to $J_{g-1,2}$ under a translation
    \begin{equation*}
        t:J_{g-1,2}\xrightarrow{\cong} J_{g-1,2}^{-1}, \quad (C,x_1,x_2,L)\mapsto (C,x_1,x_2,L(-x_2)).
    \end{equation*} 
The diagram can be deduced from the description of the torus rank~$1$ boundary of the moduli space of principally polarized abelian varieties in \cite{MumfordAb}.

\begin{prop}\label{pro:4.1}
    On the relative compactified Jacobian $\overline{J}_g^{\leq 1}$, we have 
    \begin{equation*}
        \frac{\theta^{g+1}}{(g+1)!} = \frac{1}{48}\iota_*\epsilon_*t_*\frac{\theta^{g-1}}{(g-1)!} + \iota_*\epsilon_*t_*\alpha \in \Chow^{g+1}(\overline{J}_g^{\leq 1},\BQ)
    \end{equation*}
    for some class $\alpha\in \bigoplus_{w<2g-2}\Chow^{g-1}_{(w)}(J_{g-1, 2}, \BQ)$. Here $w$ stands for the weight with respect to the ``multiplication by $N$'' map on $J_{g-1, 2}$.
\end{prop}

\begin{proof}
    Recall the universal double ramification cycle relation\footnote{In general the universal double ramification cycle relation depends on a vector $A$ of integers. Here we only consider the special case $A=\emptyset$.} \cite[Theorem 8]{BHPSS} on the universal Picard stack:
    \begin{equation*}
        \mathsf{P}^{g+1}_g := \mathsf{P}^{g+1}_{g,A=\emptyset} = 0 \in \Chow^{g + 1}_{\mathsf{op}}(\ppic_g,\BQ).
    \end{equation*}
     We restrict Pixton's formula for $\mathsf{P}_g^{g+1}$ to the relative Picard stack over $\mg$; we thus focus on the substack $\mathfrak{M}_g^{\leq 1}\subset\mathfrak{M}_g$ where the quasi-stable curve stabilizes to a stable curve with at most one self-node. Let $i:\ppic_\delta^\rel\to\ppic^\rel_g$ be the inclusion of the boundary strata where the underlying curve has at least one node.  For an integer $d$, let $i^d:\ppic^\rel_{\Gamma_d}\to\ppic^\rel_g$ be the boundary stratum corresponding to a decorated prestable graph $\Gamma_d$ with two vertices of degree splitting $d,-d$, connected by two edges $e_1=(h_1,h_1')$ and $e_2=(h_2,h_2')$. Then Pixton's formula has the form
    \begin{multline} \label{eq:uniDR}
        \mathsf{P}^{g+1}_g = \frac{\theta^{g+1}}{(g+1)!}+\sum_{a+b=g}\mathrm{cst}\cdot \theta^a i_*\big((\psi_h+\psi_{h'})^b\big) \\
        {}+ \sum_{d\in\BZ}\, \sum_{k+\ell+m=g-1} f_{k,\ell,m}(d)\cdot \theta^k i^d_*\big((\psi_{h_1}+\psi_{h_1'})^\ell(\psi_{h_2}+\psi_{h_2'})^{m}\big),
    \end{multline}
    where $f_{k,\ell, m}$ is a polynomial with rational coefficients of degree at most $2(\ell+m+2)$. Here the polynomiality of $f_{k,\ell,m}$ follows from Faulhaber's formula on sum of powers applied to Pixton's formula. We refer to \cite[Section 0.3.5]{BHPSS} for notations, and explicit expressions of cst and $f_{k,\ell,m}$. 
    
    For our purpose, we further refine the relation $\mathsf{P}_g^{g + 1} = 0$ using the ``multiplication by~$N$''~map
    \[
    [N]:\ppic_g^\rel\to \ppic_g^\rel.
    \]
    For the first two terms of \eqref{eq:uniDR}, we have \[
    [N]^*\theta^{g+1} = N^{2g+2}\cdot\theta^{g+1}
    \]
    and
    \[[N]^*\left(\theta^a i_*\big((\psi_h+\psi_{h'})^b\big)\right) = N^{2a}\cdot \theta^a i_*\big((\psi_h+\psi_{h'})\big)^b\]
    respectively. For the third term of \eqref{eq:uniDR}, we have
    \[[N]^*\left(\theta^ki^d_*\big((\psi_{h_1}+\psi_{h_1'})^\ell(\psi_{h_2}+\psi_{h_2'})^{m}\big)\right) = N^{2k}\cdot \theta^ki^{d/N}_*\big((\psi_{h_1}+\psi_{h_1'})^\ell(\psi_{h_2}+\psi_{h_2'})^{m}\big)\]
    if $N$ divides $d$ and zero otherwise. Applying $[N]^*$ to the right-hand side of \eqref{eq:uniDR}, the first two terms are clearly polynomials in $N$ and for the third term, we have
    \begin{align*}
        &[N]^* \sum_{d\in\BZ}\, \sum_{k+\ell+m=g-1} f_{k,\ell,m}(d)\cdot \theta^k i^d_*\big((\psi_{h_1}+\psi_{h_1'})^\ell(\psi_{h_2}+\psi_{h_2'})^{m}\big)\\ 
        & = \sum_{N|d}\, \sum_{k+\ell+m=g-1} f_{k,\ell,m}(d)\cdot N^{2k}\cdot\theta^k i^{d/N}_*\big((\psi_{h_1}+\psi_{h_1'})^\ell(\psi_{h_2}+\psi_{h_2'})^{m}\big)\\
        & = \sum_{d\in\BZ}\, \sum_{k+\ell+m=g-1} f_{k,\ell,m}(Nd)\cdot N^{2k}\cdot\theta^k i^d_*\big((\psi_{h_1}+\psi_{h_1'})^\ell(\psi_{h_2}+\psi_{h_2'})^{m}\big)
    \end{align*}
    which is a polynomial in $N$ because $f_{k,\ell,m}(d)$ are polynomials.
    Therefore we find an expression of $[N]^*\mathsf{P}^{g+1}_g = 0$ which is a polynomial in $N$, so that each coefficient gives rise to a relation on $\ppic^\rel_g$. We take the coefficient of $N^{2g+2}$ in the expression of $[N]^*\mathsf{P}_g^{g+1} = 0$ and pull it back to $\overline{J}_g^{\leq 1}$.  In particular, the second summand of the right-hand side of \eqref{eq:uniDR} does not contribute for weight reasons. This gives the desired relation on~$\overline{J}_g^{\leq 1}$ of the proposition.  
    
  For the reader's convenience, we spell out in the following more details in deducing this relation. The contribution from the prestable graph with two vertices of
    \begin{enumerate}
        \item[$\bullet$]  genus $g-1$ with degree $-1$, and 
        \item[$\bullet$] genus $0$ with degree $1$         
    \end{enumerate}
    connected by two edges corresponds to a class under the pushforward $\iota_*\epsilon_*t_*$. We denote the two edges by $e_1=(h_1,h_1')$ and $e_2=(h_2,h_2')$ where $h_1,h_2$ are the two half-edges attached to the genus~$g-1$ vertex.  After pulling back to $\overline{J}_g^{\leq 1}$, other strata do not contribute by the quasi-stable model of a rank $1$ torsion-free sheaf on a nodal curve. Then we have
    \begin{equation}\label{eq33}
    \frac{\theta^{g+1}}{(g+1)!} = \frac{1}{48}\iota_*\epsilon_*t_*\frac{\theta^{g-1}}{(g-1)!} + \iota_*\epsilon_*t_*(\text{other terms}).\end{equation}
    Here the coefficient $\frac{1}{48}$ associated with the leading term of the right-hand side can be read of from the explicit computation of the polynomial: 
    \[
    f_{g-1,0,0}(d) = -\frac{d^4}{48} +\frac{d^2}{24} - \frac{1}{240}.
    \]
    The other terms correspond to a linear combination of prestable graphs decorated by monomials of $\theta$ and $\psi$-classes. The $\psi$-classes are the contributions of the first Chern class of the cotangent line bundle at the various markings (see also Section \ref{Sec4.3}). 
    
    We explain how to interpret the classes given by prestable graphs decorated by $\psi$-monomials as classes pushed forward from $J_{g-1,2}$. 
    Let $\CL$ be the universal line bundle on the universal curve over $J_{g-1, 2}$ trivialized along the first marking, and let $\xi_2\in\Chow^1_{(1)}(J_{g-1, 2}, \BQ)$ be $c_1(\CL)$ pulled back along the second marking. Analogous classes are defined on $J_{g-1,2}^{-1}$.
    We take the coefficient of~$N^1$ in the expression of $[N]^*\mathsf{P}^1_{0,(1,-1)} = 0$. This gives a relation
    \[ \psi_{h'_2} = -\psi_{h'_1}= \xi_{h'_1}-\xi_{h'_2}\]
    at the unstable genus $0$ vertex. Since $\xi_{h'_1}=\xi_{h_1}=0$ and $\xi_{h'_2}=\xi_{h_2}=\xi_2$, the ``other terms'' in~(\ref{eq33}) have an expression in terms of monomials of $\theta, \psi_\bullet,\xi_2$. Under the translation, these monomials pull back to classes on $J_{g-1,2}$ of weights \mbox{$< 2g-2$}. \qedhere

\end{proof}
\begin{rmk}
    Although the ``multiplication by $N$" map extends to the relative Picard stack $\ppic^\rel_g$, the cohomology $H^*(\ppic^\rel_g,\BQ)$ may not be a direct sum of weight spaces since it is infinite-dimensional. 
\end{rmk}

We denote by $J^{\leq 1}_g \subset \overline{J}^{\leq 1}_g$ the open subset parameterizing degree $0$ line bundles. We have the tautological relation \[
\theta^{g+1}=0 \in \mathrm{CH}^{g + 1}(J^{\leq 1}_g, \BQ)
\]
by \cite[Theorem 4.9]{AHPL}. As a consequence of the proposition above, we show that this vanishing fails to be further extended to the relative compactified Jacobian $\overline{J}_g^{\leq 1}$.
\begin{cor}\label{cor:4.2}
    For $g\geq 2$, we have
    \[    \pi_*\frac{\theta^{g+1}}{(g+1)!}  = \frac{1}{48}\iota_*[\CM_{g-1,2}] \in \mathrm{CH}^1(\overline{\CM}_g^{\leq 1},\BQ).
    \]
In particular we have the non-vanishing
\[
\theta^{g+1}\neq 0 \in H^{2g + 2}(\overline{J}_g^{\leq 1},\BQ).
\]
\end{cor}
\begin{proof}
    By Proposition \ref{pro:4.1}, it is enough to calculate the pushforward along $\pi$ of the classes $\iota_*\epsilon_*t_*(\theta^{g-1}/(g-1)!)$ and $\iota_*\epsilon_*t_*\alpha$. Since \mbox{$\alpha\in \mathrm{CH}^{g - 1}(J_{g-1, 2},\BQ)$} has weights $<2g-2$, we have $\pi_*\iota_*\epsilon_*t_*\alpha = \iota_*\pi'_*t_*\alpha=0$. We also recall 
    \begin{equation}\label{eq:unit}
        \frac{\theta^{g-1}}{(g-1)!}=[0] \in \mathrm{CH}^{g - 1}(J_{g-1, 2},\BQ),
    \end{equation}
    where $0:\CM_{g-1,2}\to J_{g-1, 2}$ is the $0$-section; see \emph{e.g.}~\cite{DM}. Therefore by the diagram \eqref{eq:Jbar} we~have
    \[
    \pi_*\iota_*\epsilon_*t_*\frac{\theta^{g-1}}{(g - 1)!} = \iota_*[\CM_{g-1,2}];
    \]
    it is proportional to the fundamental class of the boundary stratum which is nontrivial in the cohomology $H^2(\mg,\BQ)$.
\end{proof}

\begin{rmk}
Here the universal double ramification relation of Proposition \ref{pro:4.1} helps us to relate $\theta^{g+1}$ to an explicit nontrivial boundary class. 
\end{rmk}

\subsection{Proof of Theorem \ref{thm4.1}}\label{Sec4.3}
In this section we prove some non-vanishing results on the cohomology of the relative compactified Jacobian which imply Theorem \ref{thm4.1}.

We recall a few standard tautological classes. The cotangent line of the $i$-th marking defines a line bundle $\BL_i$ on $\overline{\CM}_{g,n}$ with first Chern class 
\[
\psi_i : = c_1(\BL_i) \in H^2(\overline{\CM}_{g,n}, \BQ).
\]
The forgetful map $p: \overline{\CM}_{g,1} \to \overline{\CM}_g$ coincides with the universal family of curves, and we define the $i$-th $\kappa$-class to be
\[
\kappa_i: = p_*\psi_1^{i+1} \in H^{2i}(\overline{\CM}_{g}, \BQ).
\]
After restriction, we have the $\kappa$-classes on $\mg$. We also define the boundary class
\[
\delta:= \iota_*[\CM_{g-1,2}]\in H^2(\mg, \BQ).
\]
By the pullback $\pi^*$, these tautological classes are naturally viewed as classes in $H^*(\overline{J}_g^{\leq 1},\BQ)$.

We first note that the degree $2$ cohomology group of $\overline{J}_g^{\leq 1}$ is tautological.

\begin{lem}\label{lem:4.1}
    For $g\geq 2$, $H^2(\overline{J}_g^{\leq 1}, \BQ)$ is spanned by the tautological classes 
    \[
    \theta, \kappa_1, \delta \in H^2(\overline{J}_g^{\leq 1}, \BQ).
    \]
\end{lem}

\begin{proof}
    We consider $J_g^{\leq 1} \subset \overline{J}_g^{\leq 1}$ the open locus of degree $0$ line bundles.  Since its complement has (complex) codimension $2$, the restriction $H^2(\overline{J}_g^{\leq 1},\BQ)\to H^2(J_g^{\leq 1},\BQ)$ is an isomorphism. 
    
    We have the long exact sequence
    \[
    \cdots\to H^0(J_g^{\leq 1}\setminus J_g,\BQ)\to H^2(J_g^{\leq 1},\BQ)\to H^2(J_g,\BQ)\to H^1(J_g^{\leq 1}\setminus J_g,\BQ)\to \cdots.
    \]
    By \cite{ERW}, we know that $H^2(J_g,\BQ)$ is spanned by $\theta$ and $\kappa_1$. The lemma follows since the image of $H^0(J_g^{\leq 1}\setminus J_g,\BQ)$ is spanned by $\delta$.
\end{proof}

    

We compare the theta divisors on $\overline{J}^{\leq 1}_g$ and $J_{g-1,2}^{-1}$. Recall the tautological class $\xi_2$ from the proof of Proposition \ref{pro:4.1}.

 \begin{lem}\label{lem:pulltheta}
    Let $\CL$ be the universal line bundle on the universal curve  $p:\mathcal{C}_{g-1,2}\to J_{g-1,2}^{-1}$ which is trivialized along the first marking, which induces $\theta:= -\frac{1}{2}p_*(c_1(\CL)^2)$ on $J^{-1}_{g-1,2}$. Then we have \[
    \epsilon^*\iota^*\theta = (\theta-\xi_2)+\frac{1}{2}(\psi_1+\psi_2) \in H^2(J^{-1}_{g-1,2}, \BQ).
\]
\end{lem}
\begin{proof}
     Let $\Bar{p}:\overline{\CC}_g\to J_{g-1,2}^{-1}$ be the universal curve on $\overline{J}_g^{\leq 1}$ pulled back to $J^{-1}_{g-1,2}$ along $\iota\circ\epsilon$, and let $q:\overline{\CC}^{\mathrm{qs}}_g\to\overline{\CC}_g$ be the quasi-stable model. Let $\CL^\mathrm{qs}$ be the universal admissible line bundle on $\overline{\CC}^\mathrm{qs}_g$ trivialized along the first marking. The normalization of $\overline{\CC}^\mathrm{qs}_g$ is the disjoint union of a $\BP^1$-bundle $\BP$ over $J_{g-1,2}$ and the universal curve $\CC_{g-1,2}$. By \cite[eq.~(8)]{EP}, we have a short exact sequence
     \[0\to \CL^\mathrm{qs}|_{\CC_{g-1,2}}(-x_1-x_2) \to\CL^\mathrm{qs}\to \CL^{\mathrm{qs}}|_{\BP}\to 0\,. \]
     Since $\CL^\mathrm{qs}|_{\CC_{g-1,2}}\cong\CL$, we get
     \[\epsilon^*\iota^*\theta = -\frac{1}{2}\Bar{p}_*(c_1(\CL^\mathrm{qs})^2) = -\frac{1}{2}p_*(c_1(\CL(-x_1-x_2)^2)) = \theta-\xi_2+\frac{1}{2}(\psi_1+\psi_2)\]
     where we used the self-intersection formula $[x_i]^2=x_{i*}(-\psi_i)$.
\end{proof}
Now we treat the case when $g\geq 4$. 
\begin{prop}\label{pro:4.2}
Let   \[
    {\Theta} := \theta + a\kappa_1 + b\delta \in H^2(\overline{J}_g^{\leq 1},\BQ), \quad a,b\in \BQ
    \]
    be a generalized theta divisor. 
    \begin{enumerate}
        \item[(a)] If $g\geq 2$, then $\Theta$ can be written as $\Theta = \theta+ b\delta$.
        \item[(b)] If $g\geq 4$, then there is no such $b\in \BQ$. \end{enumerate} 
\end{prop}

\begin{proof}
   To prove (a), we use again the diagram \eqref{eq:Jbar}. First, we pull the relation $\Theta^{g + 1} = 0$ back to $J_g$ and obtain
    \[
    0=j^*{\Theta}^{g+1} = (\theta+ a\kappa_1)^{g+1} \in H^{2g+2}(J_g,\BQ).
    \]
    By taking the weight $2g$ part with respect to the ``multiplication by $N$'' map, we get 
    \[
    a\cdot\theta^g\kappa_1=0 \in H^{2g+2}(J_g, \BQ).
    \]
    Taking pushforward and applying \eqref{eq:unit}, we have
    \[
    a\kappa_1=0 \in H^2(\CM_g, \BQ).
    \]
    This shows that $a\kappa_1$ is supported on the boundary $\overline{\CM}_g^{\leq 1} \setminus \CM_g$ and is therefore proportional to $\delta$ by the irreducibility of the boundary.\footnote{If $g \geq 3$, we can directly deduce $a = 0$ from $\kappa_1 \neq 0 \in H^2(\CM_g, \BQ)$.}
    
    To prove (b), we restrict the relation to $J_{g-1, 2}$. Let $\psi_1,\psi_2$ be the pullback of the $\psi$-classes on $\CM_{g-1,2}$ to $J_{g-1, 2}$.  
    By Lemma \ref{lem:pulltheta} and \cite[Lemma 4.2]{BHPSS}, we have
    \[t^*\epsilon^*\iota^*\theta = \theta+\frac{1}{2}(\psi_1+\psi_2).\]
     Since the excess intersection term for $\iota$ is $-(\psi_1+\psi_2)$, we have
    \[
    0=t^*\epsilon^*\iota^*{\Theta}^{g+1} = (\theta +(\frac{1}{2}-b)(\psi_1+\psi_2))^{g+1} \in H^{2g + 2}(J_{g-1, 2},\BQ).
    \]
    By taking the weight $2g-2$ part with respect to the ``multiplication by $N$'' map, we get 
    \[
    (\frac{1}{2} - b)^2\cdot \theta^{g-1}(\psi_1+\psi_2)^2=0 \in H^{2g+2}(J_{g-1, 2}, \BQ).
    \]
    Taking pushforward and using \eqref{eq:unit}, we have $(\frac{1}{2} - b)^2 (\psi_1+\psi_2)^2=0$ on $\CM_{g-1,2}$. If $g\geq 4$, we have the non-vanishing 
    \begin{equation}\label{non-v}
    (\psi_1+\psi_2)^2\neq 0 \in H^4(\CM_{g-1,2},\BQ)
    \end{equation}
    by the calculation of top intersection numbers \cite{BSZ}. Therefore $b=\frac{1}{2}$.

    On the other hand, by Corollary \ref{cor:4.2} we have
    \[0 = \pi_*\left(\frac{\Theta^{g+1}}{(g+1)!}\right) = \pi_*\left(\frac{\theta^{g+1}}{(g+1)!}+b\frac{\theta^g}{g!}\delta\right) = \pi_*\left(\frac{1}{48}\iota_*\epsilon_*t_*\frac{\theta^{g-1}}{(g-1)!} +b\frac{\theta^g}{g!}\delta\right) = (\frac{1}{48}+b)\delta.\]
    Since $\delta\neq 0$ in $H^2(\mg,\BQ)$, we have $b = -\frac{1}{48}$ which is a contradiction.
\end{proof}

By Proposition \ref{pro:4.2}, we conclude Theorem \ref{thm4.1} for $g\geq 4$. However, the non-vanishing (\ref{non-v}) fails for $g=2,3$. Therefore we treat them separately in the following.


\medskip
\noindent {\bf $\bullet$ The $g=3$ case.} Let ${\Theta}$ be a generalized theta divisor. By Proposition~\ref{pro:4.2}(a) we set
\[
{\Theta}= \theta + b\delta, \quad b \in \BQ.
\]
Recall the class $\alpha\in\Chow^2(J_{2, 2},\BQ)$ of Proposition \ref{pro:4.1}. We write the weight decomposition
\[
\alpha=\sum\alpha_{(w)}, \quad  \alpha_{(w)}\in\Chow^2_{(w)}(J_{2, 2}, \BQ).
\] 
By an explicit computation on the relative Jacobian $J_{2, 2}$, we have 
\[
\alpha_{(2)} = \frac{1}{480}\theta(\psi_1+\psi_2)-\frac{1}{8960}\xi_2^2.
\]
Here $\xi_2$ was introduced in the proof of Proposition \ref{pro:4.1}.

Since we have the relation $\delta^3=0$ for $g=3$ by \cite{Faber}, we get from Proposition~\ref{pro:4.1} that
\begin{align*}
    0 &=\pi_*(\theta{\Theta}^4)\\
    &= \pi_*\left(\iota_*\epsilon_*t_*\Big(\frac{1}{4}(\theta+\frac{1}{2}(\psi_1+\psi_2))^3+ 24\alpha_{(2)}\theta - b(\theta+\frac{1}{2}(\psi_1+\psi_2))^2(\psi_1+\psi_2)\Big)+6b^2\theta^3\delta^2\right) \\
    &=\left(\frac{191}{224} - 2b -36b^2\right)\iota_*(\psi_1+\psi_2) \in H^4(\overline{\CM}_3^{\leq 1}, \BQ).
\end{align*}
Here we used the tautological relation $2\theta\xi_2^2=-\theta^2(\psi_1+\psi_2)\in H^6(J_{2, 2},\BQ)$.
Note that 
\[
\psi_1+\psi_2\neq 0 \in H^2(\CM_{2,2},\BQ),\quad H^3(\CM_3,\BQ)=0
\]
where the second vanishing can be seen from \cite{Loo}. Applying the long exact sequence of cohomology associated with the open embedding $\CM_3 \hookrightarrow \overline{\CM}_3^{\leq 1}$, we find
\[
\iota_*(\psi_1+\psi_2)\neq 0 \in H^4(\overline{\CM}_3^{\leq 1}, \BQ).
\]
This forces $b$ to satisfy a quadratic equation
\[
\frac{191}{224} - 2b -36b^2=0
\]
which does not have a rational solution. \qed 

\medskip
\noindent {\bf $\bullet$ The $g=2$ case.}
For $g=2$, note that there actually exists a generalized theta divisor
\[
{\Theta} := \theta - \frac{1}{48} \delta \in H^2(\overline{J}_2^{\leq 1},\BQ).
\]
Here the coefficient $-\frac{1}{48}$ can be uniquely determined by the relation 
\[
\pi_* ({\Theta}^3) = 0 \in H^2(\overline{\CM}_2^{\leq 1},\BQ).
\]
Therefore, we have to consider the larger open subset $\overline{\CM}^{\mathrm{int}}_2 \subset \overline{\CM}_2$ consisting of integral curves  and the relative compactified Jacobian $\pi^\mathrm{int}: \overline{J}^\mathrm{int}_2\to\overline{\CM}^\mathrm{int}_2$ as in Theorem \ref{thm4.1}(b).

Since the codimension of $\mtwo\setminus \overline{\CM}_2^{\leq 1}$ is $2$, Proposition \ref{pro:4.2}(a) shows that a generalized theta divisor has to be of the form 
\[
{\Theta}= \theta+b\delta, \quad b\in \BQ.
\]  
We consider the open subset 
\[
\overline{\CM}_{1,2}^\circ = \overline{\CM}_{1,2}\setminus \left(\overline{\CM}_{1,1}\times\overline{\CM}_{0,3} \right) \subset \overline{\CM}_{1,2}
\]
with the gluing map 
\[
\iota^{\mathrm{int}}:\overline{\CM}_{1,2}^\circ \to \mtwo.
\]
Let $r^{\mathrm{int}}: \overline{\CM}_{0,4}\setminus D_{1,2|3,4}\to \mtwo$ be the gluing map with respect to the deeper stratum. Here~$D_{1,2|3,4}$ is the locus where the markings split into two pairs of points on the two rational components. We have by \cite{Mumford} the tautological relations 
\[
\delta^3=0, \quad \delta^2 =-\frac{1}{6}r^{\mathrm{int}}_*\left[ \overline{\CM}_{0,4}\setminus D_{1,2|3,4}\right]\in H^*(\overline{\CM}^{\mathrm{int}}_2, \BQ).
\]

On $\overline{J}^\mathrm{int}_2$, the universal double ramification cycle relation $\mathsf{P}^3_2=0$ contains contribution from the boundary strata with two loops. However, those boundary strata contribute to lower degree coefficients of the expression $[N]^*\mathsf{P}^3_2=0$. Therefore Proposition \ref{pro:4.1} continues to hold on~$\overline{J}^\mathrm{int}_2$.
A direct computation yields  $\alpha_{(0)}=\frac{1}{480}(\psi_1+\psi_2)$. Combining with Proposition \ref{pro:4.1} and Lemma~\ref{lem:pulltheta}, we obtain
\begin{align*}
    0&=\pi^\mathrm{int}_*(\theta{\Theta}^3) 
    = \pi^{\mathrm{int}}_*\left(\iota^{\mathrm{int}}_*\epsilon_*t_*\Big(\frac{1}{8}(\theta+\frac{1}{2}(\psi_1+\psi_2))^2+ 6\theta\alpha_{(0)}-\frac{3b}{8}\theta(\psi_1+\psi_2)\Big)+3b^2\theta^2\delta^2\right) \\
    & =\iota^\mathrm{int}_*\left(\frac{1}{8}(\psi_1+\psi_2)+6\alpha_{(0)}-\frac{3b}{8}(\psi_1+\psi_2)\right) +6b^2\delta^2\\
    & = \left(\frac{11}{960}-\frac{1}{32}b-b^2\right)r^{\mathrm{int}}_*\left[ \overline{\CM}_{0,4}\setminus D_{1,2|3,4}\right] \in H^4(\mtwo, \BQ).
\end{align*}
Then, due to the non-vanishing
\[
r^{\mathrm{int}}_*\left[ \overline{\CM}_{0,4}\setminus D_{1,2|3,4}\right]\neq 0 \in H^4(\mtwo,\BQ),
\]
we have  
\[
\frac{11}{960}-\frac{1}{32}b-b^2=0
\]
which is a contradiction.
     
This completes the proof of Theorem \ref{thm4.1}(b). \qed


\subsection{Deligne--Mumford stacks versus varieties}\label{final}

So far we have worked universally over the moduli space of stable curves, and the base in our example is a nonsingular Deligne--Mumford stack. We explain here that our arguments actually produce examples with nonsingular quasi-projective base varieties; this proves \emph{honestly} Theorem \ref{thm0.40}.

To start with, we take a smooth and surjective morphism 
\[
\rho: \overline{B} \to \overline{\CM}_g
\]
such that $\overline{B}$ is a (nonsingular) projective variety and $\rho$ has irreducible fibers; here the existence of $\rho$ follows from the {GIT} construction of the moduli of stable curves. In particular, the pullback map
\begin{equation}\label{pullback0} 
\rho^*: H^*(\overline{\CM}_g, \BQ) \to H^*(\overline{B}, \BQ)
\end{equation}
is injective. We write 
\[
B^\circ:= \rho^{-1}(\CM_g) \subset \overline{B}.
\]
If we pull back the compactified Jacobian fibration of Theorem \ref{thm4.1} along $\rho$, we get a compactified Jacobian fibration 
\[
\pi: \overline{J}_C \to B, \quad B^\circ \subset B \subset \overline{B}.
\]
By repeating the proof of Proposition \ref{pro:4.2}(a) using the irreducibility of the boundary $B \setminus B^\circ$, up to scaling a generalized theta divisor on $\overline{J}_C$ can be expressed as
\[
\Theta = \rho^*\theta  + b \delta_B, \quad \delta_B:= [B\setminus B^\circ] = \rho^*\delta.
\]
In other words, a generalized theta divisor on $\overline{J}_C$ has to be pulled back from the moduli stack of stable curves. Theorem \ref{thm4.1}, combined with the injectivity of (\ref{pullback0}), then implies that it cannot happen. We have completed the proof of Theorem \ref{thm0.40}.\qed

\end{document}